\documentclass[11pt]{article}
\RequirePackage{amsthm,amsmath, bm}
\usepackage{graphicx,psfrag,epsf}
\usepackage{enumerate}
\usepackage{natbib}
\usepackage{url} 
\usepackage{subfigure}

\newcommand{\blind}{0}

\newcommand{\pr}{\text{pr}}
\newcommand{\var}{\text{var}}
\newcommand{\Var}{\text{var}}
\newcommand{\cov}{\text{cov}}
\newcommand{\sumN}{\sum_{i=1}^N}

\newcommand{\convD}{ \stackrel{d}{\longrightarrow} }
\def\obs{\textnormal{obs}}

\theoremstyle{definition}

\newtheorem{theorem}{Theorem}
\newtheorem{lemma}{Lemma}
\newtheorem{example}{Example}

\newtheorem{corollary}{Corollary}

\begin{document}

\def\spacingset#1{\renewcommand{\baselinestretch}
{#1}\small\normalsize} \spacingset{1}


\if0\blind
{
  \title{\bf A Paradox From Randomization-Based Causal Inference}
  \author{Peng Ding
  \thanks{
Peng Ding, Department of Statistics, University of California at Berkeley, 425 Evans Hall, Berkeley, California 94720 USA (E-mail: pengdingpku@berkeley.edu).   
I want to thank Professors Donald Rubin, Arthur Dempster, Tyler VanderWeele, James Robins, Alan Agresti, Fan Li, Peter Aronow, Sander Greenland and Judea Pearl for their comments. 
Dr. Avi Feller at Berkeley, Dr. Arman Sabbaghi at Purdue, and Misses Lo-Hua Yuan and Ruobin Gong at Harvard helped edit early versions of this paper. 
I am particularly grateful to Professors Tirthankar Dasgupta and Luke Miratrix for their continuous encouragement and help during my writing of this paper.  A group of Harvard undergraduate students, Taylor Garden, Jessica Izhakoff and Zoe Rosenthal, collected the data from a $2^4$ full factorial design for the final project of Professors Dasgupta and Rubin's course ``Design of Experiments'' in Fall, 2014. They kindly shared their interesting data with me. 
Based on an early version of this paper, I received the 2014 Arthur P. Dempster Award from the Arthur P. Dempster Fund of the Harvard Statistics Department, established by Professor Stephen Blyth. 
I am also grateful to the detailed technical comments from one reviewer and many helpful historical comments from the other reviewer.
}  
 }
\date{}  
  \maketitle
} \fi

\if1\blind
{
   \title{\bf A Paradox from Randomization-Based Causal Inference}
   \date{}  
  \maketitle
} \fi

\bigskip
\begin{abstract}
Under the potential outcomes framework, causal effects are defined as comparisons between potential outcomes under treatment and control. To infer causal effects from randomized experiments, Neyman proposed to test the null hypothesis of zero average causal effect (Neyman's null), and Fisher proposed to test the null hypothesis of zero individual causal effect (Fisher's null). Although the subtle difference between Neyman's null and Fisher's null has caused lots of controversies and confusions for both theoretical and practical statisticians, a careful comparison between the two approaches has been lacking in the literature for more than eighty years. We fill in this historical gap by making a theoretical comparison between them and highlighting an intriguing paradox that has not been recognized by previous researchers. Logically, Fisher's null implies Neyman's null. It is therefore surprising that, in actual completely randomized experiments, rejection of Neyman's null does not imply rejection of Fisher's null for many realistic situations, including the case with constant causal effect. Furthermore, we show that this paradox also exists in other commonly-used experiments, such as stratified experiments, matched-pair experiments, and factorial experiments. Asymptotic analyses, numerical examples, and real data examples all support this surprising phenomenon. Besides its historical and theoretical importance, this paradox also leads to useful practical implications for modern researchers.
\end{abstract}

\noindent%
\vfill

\newpage
\spacingset{1.45} 

\section{Introduction}
Ever since Neyman's seminal work, the potential outcomes framework \citep{neyman::1923, rubin::1974} has been widely used for causal inference in randomized experiments \citep[e.g.,][]{neyman::1935, hinkelmann::2007, imbens::2015}. The potential outcomes framework permits making inference about a finite population of interest, with all potential outcomes fixed and randomness coming solely from the physical randomization of the treatment assignments.
Historically, \citet{neyman::1923} was interested in obtaining an unbiased estimator with a repeated sampling evaluation
of the average causal effect, which corresponded to a test for the null hypothesis of zero average causal effect.
On the other hand, \citet{fisher::1935a} focused on testing the sharp null hypothesis of zero individual causal effect, and proposed the Fisher Randomization Test (FRT).
Both Neymanian and Fisherian approaches are randomization-based inference, relying on the physical randomization of the experiments.
Neyman's null and Fisher's null are closely related to each other:
the latter implies the former, and they are equivalent under the constant causal effect assumption.
Both approaches have existed for many decades and are widely used in current statistical practice. They are now introduced at the beginning of many causal inference courses and textbooks \citep[e.g.,][]{rubin::2004, imbens::2015}. Unfortunately, however, a detailed comparison between them has not been made in the literature.

In the past, several researchers \citep[e.g.,][page 40]{rosenbaum::2002} believed that ``in most cases, their disagreement is entirely without technical consequence: the same procedures are used, and the same conclusions are reached.''
However, we show, via both numerical examples and theoretical investigations, that 
the rejection rate of Neyman's null is higher than that of Fisher's null
in many realistic randomized experiments, using their own testing procedures.
In fact, Neyman's method is always more powerful if there is a nonzero constant causal effect, the very alternative most often used for Fisher-style inference.
This finding immediately causes a seeming paradox:
logically, Fisher's null implies Neyman's null, so how can we fail to reject the former while rejecting the latter?

We demonstrate that this surprising paradox is not unique to completely randomized experiments, because it also exists in other commonly-used experiments such as stratified experiments, matched-pair experiments, and factorial experiments.
The result for factorial experiments helps to explain the surprising empirical evidence in \cite{dasgupta::2012} that interval estimators for factorial effects obtained by inverting a sequence of FRTs are often wider than Neymanian confidence intervals.

The paper proceeds as follows.
We review Neymanian and Fisherian randomization-based causal inference in Section \ref{sec::randomization-inference} under the potential outcomes framework. In Section \ref{sec::paradox-neyman-fisher}, we use both numerical examples and asymptotic analyses to demonstrate the paradox from randomization-based inference in completely randomized experiments. 
Section \ref{sec::other-experiments} shows that a similar paradox also exists in other commonly-used experiments. 
Section \ref{sec::extensions} extends the scope of the paper to improved variance estimators and comments on the choices of test statistics. 
Section \ref{sec::examples} illustrates the asymptotic theory of this paper with some finite sample real-life examples.
We conclude with a discussion in Section \ref{sec::discussion}, and relegate all the technical details to the Supplementary Material.

\section{Randomized Experiments and Randomization Inference}
\label{sec::randomization-inference}
We first introduce notation for causal inference in completely randomized experiments,
and then review the Neymanian and Fisherian perspectives for causal inference.

\subsection{Completely Randomized Experiments and Potential Outcomes}

Consider $N$ units in a completely randomized experiment. 
Throughout our discussion, we make the Stable Unit Treatment Value Assumption \citep[SUTVA;][]{cox::1958, rubin::1980}, i.e., there is only one version of the treatment, and interference between subjects is absent. 
SUTVA allows us to define the potential outcome of unit $i$ under treatment $t$ as $Y_i(t)$, with $t=1$ for treatment and $t=0$ for control.
The individual causal effect is defined as a comparison between two potential outcomes, for example, $\tau_i = Y_i(1) -  Y_i(0)$. 
However,
for each subject $i$, we can observe only one of $Y_i(1)$ and $Y_i(0)$ with the other one missing, and the individual causal effect $\tau_i$ is not observable.
The observed outcome is a deterministic function of the treatment assignment $T_i$ and the potential outcomes, namely, $Y_i^{\obs} = T_i Y_i(1) + (1 - T_i) Y_i(0)$. Let $\bm{Y}^{\obs}= (Y_1^{\obs}, \ldots, Y_N^{\obs}) ' $ be the observed outcome vector. 
Let $\bm{T} = (T_1, \ldots, T_N) ' $ denote the treatment assignment vector, and $\bm{t} = (t_1, \ldots, t_N) ' \in \{0, 1\}^N$ be its realization.
Completely randomized experiments satisfy
$
\pr\left( \bm{T} = \bm{t}     \right) = N_1 ! N_0 ! / N !,
$
if $  \sumN t_i = N_1$ and $N_0 = N-N_1$.
Note that in \citet{neyman::1923}'s potential outcomes framework, all the potential outcomes are fixed numbers, and only the treatment assignment vector is random. In general, we can view this framework with fixed potential outcomes as conditional inference given the values of the potential outcomes.
In the early literature, \citet{neyman::1935} and \citet{kempthorne::1955} are two research papers, 
and \citet{kempthorne::1952}, \citet[][Chapter 9]{hodges::1960} and \citet[][Chapter 9]{scheffe::1959} are three textbooks using potential outcomes for analyzing experiments.

\subsection{Neymanian Inference for the Average Causal Effect}
\label{sec::neyman}
\citet{neyman::1923} was interested in estimating the finite population average causal effect:
$$
\tau    = \frac{1}{N} \sumN \tau_i  =  \frac{1}{N} \sumN \{  Y_i(1) - Y_i(0) \}  =  \bar{Y}_1 - \bar{Y}_0,
$$
where $\bar{Y}_t  = \sumN Y_i(t)/N$ is the finite population average of the potential outcomes $\{ Y_i(t): i=1, \ldots, N\}$. He proposed
an unbiased estimator 
\begin{eqnarray}\label{eq::tau-hat}
\widehat{ \tau }  &=& \bar{Y}_1^{\obs} - \bar{Y}_0^{\obs}
\end{eqnarray} 
for $\tau$, 
where $\bar{Y}_t^{\obs} = \sum_{ \{ i: T_i=t \} } Y_i^{\obs} / N_t $ is the sample mean of the observed outcomes under treatment $t$.
The sampling variance of $\widehat{\tau}$ over all possible randomizations is
\begin{eqnarray}
\label{eq::variance-neymanian}
\var(   \widehat{\tau} )  = \frac{  S_1^2 } { N_1} + \frac{ S_0^2} {N_0}   
   - \frac{  S_\tau^2} { N}, 
\end{eqnarray}
depending on
$
   S_t^2  =    \sumN \{ Y_i(t) - \bar{Y}_t\}^2/(N-1),
$ 
 the finite population variance of the potential outcomes $\{ Y_i(t) :i=1, \ldots, N \}$, and 
$
   S_{\tau}^2 =  \sumN \left(   \tau_i -  \tau  \right)^2/(N-1),
$  
the finite population variance of the individual causal effects $\{ \tau_i : i=1, \ldots, N\} $.
Note that previous literature used slightly different notation for $S_\tau^2$, e.g., $S_{1\text{-}0}^2$ \citep{rubin::1990, imbens::2015}.
Because we can never jointly observe the pair of potential outcomes for each unit, the variance of individual causal effects, $S_\tau^2$, is not identifiable from the observed data.
Recognizing this difficulty, \citet{neyman::1923} suggested using
\begin{eqnarray}
\label{eq::variance-neyman-estimator}
\widehat{V}(\text{Neyman}) = 
\frac{ s_1^2} { N_1 } + \frac{ s_0^2}{ N_0},
\end{eqnarray}
as an estimator for $\var( \widehat{ \tau } ) $,
where $s_t^2 = \sum_{\{i: T_i = t  \}}  (Y_i^{\obs} - \bar{Y}_t^{\obs})^2/(N_t - 1) $ is the sample variance of the observed outcomes under treatment $t$.
However, Neyman's variance estimator overestimates the true variance, in the sense that 
$
E\{ \widehat{V}(\text{Neyman}) \}  \geq \var(  \widehat{\tau} ) ,
$
with equality holding if and only if the individual causal effects are constant: $\tau_i = \tau$ or $S_\tau^2=0.$
The randomization distribution of $\widehat{\tau}$ enables us to test the following Neyman's null hypothesis:
$$
H_0(\text{Neyman}) : \tau = 0.
$$
Under $H_0(\text{Neyman})$ and based on the Normal approximation in Section \ref{sec::normal}, the $p$-value from Neyman's approach can be approximated by
\begin{eqnarray}
p(\text{Neyman}) \approx  2  \Phi\left\{  -  \frac{   |  \widehat{\tau}^{\obs} | }{  \sqrt{ \widehat{V}(\text{Neyman}) }   } \right\},
\label{eq::pneyman}
\end{eqnarray}
where $\widehat{\tau}^{\obs}$ is the realized value of $\widehat{\tau}$, and $\Phi(\cdot)$ is the cumulative distribution function of the standard Normal distribution.
With non-constant individual causal effects, 
Neyman's test for the null hypothesis of zero average causal effect tends to be ``conservative,'' in the sense that it rejects less often than the nominal significance level when the null is true.

\subsection{Fisherian Randomization Test for the Sharp Null}
\citet{fisher::1935a} was interested in testing the following sharp null hypothesis:
$$
H_0(\text{Fisher}): Y_i(1) = Y_i(0), \quad  \forall i = 1, \ldots, N.
$$
This null hypothesis is sharp because all missing potential outcomes can be uniquely imputed under $H_0(\text{Fisher})$. The sharp null hypothesis implies that $Y_i(1)  = Y_i(0) = Y_i^{\obs}$ are all fixed constants, so that
the observed outcome for subject $i$ is $Y_i^{\obs}$ under any treatment assignment.
Although we can perform randomization tests using any test statistics capturing the deviation from the null, we will first focus on the randomization test using 
$
 \widehat{\tau} ( \bm{T},  \bm{Y}^{\obs}     ) = \widehat{\tau} 
$ 
as the test statistic, in order to make a direct comparison to Neyman's method. We will comment on other choices of test statistics in the later part of this paper.
Again, the randomness of $\widehat{\tau} ( \bm{T},  \bm{Y}^{\obs}     )$
comes solely from the randomization of the treatment assignment $\bm{T} $, because $\bm{Y}^{\obs} $ is a set of constants under the sharp null.
The $p$-value for the two-sided test under the sharp null is 
$$
p(\text{Fisher}) =   \pr\left\{     | \widehat{\tau} ( \bm{T},  \bm{Y}^{\obs}     )  |  \geq   | \widehat{\tau}^{\obs}   | ~\Big\vert~   H_0(\text{Fisher})  \right\} ,
$$ 
measuring the extremeness of $\widehat{\tau}^{\obs}$ with respect to the null distribution of $\widehat{\tau} ( \bm{T},  \bm{Y}^{\obs}     ) $ over all possible randomizations.
In practice, we can approximate the exact distribution of $\widehat{\tau} ( \bm{T},  \bm{Y}^{\obs}     ) $ by Monte Carlo. 
We draw, repeatedly and independently, completely randomized treatment assignment vectors $\{  \bm{T}^1, \ldots,  \bm{T}^M \} $,
and with large $M$ the $p$-value can be well approximated by 
$$
p(\text{Fisher}  ) \approx  \frac{1}{M} \sum_{m=1}^M I\left\{    | \widehat{\tau} ( \bm{T}^m,  \bm{Y}^{\obs}     )   |  \geq   | \widehat{\tau}^{\obs}    | \right\} .
$$

\citet{eden::1933} performed the FRT empirically, and \citet{welch::1937} and \citet{pitman::1937, pitman::1938} studied its theoretical properties. \citet{rubin::1980} first used the name ``sharp null,'' and \citet{rubin::2004} viewed the FRT as a ``stochastic proof by contradiction.'' For more discussion about randomization tests, please see \cite{rosenbaum::2002} and \citet{edgington::2007}.

\section{A Paradox From Neymanian and Fisherian Inference}
\label{sec::paradox-neyman-fisher}

Neymanian and Fisherian approaches reviewed in Section \ref{sec::randomization-inference} share some common properties but also differ fundamentally. They both rely on the distribution induced by the physical randomization, but they test two different null hypotheses and evolve from different statistical philosophies. 
In this section, we first compare Neymanian and Fisherian approaches using simple numerical examples, and highlight a surprising paradox. We then explain the paradox via asymptotic analysis.

\subsection{Initial Numerical Comparisons}\label{sec::compare}
\label{sec::numerical-examples}

We compare Neymanian and Fisherian approaches using numerical examples with both balanced and unbalanced experiments.
In our simulations, the potential outcomes are fixed, and the simulations are carried out over randomization distributions induced by the treatment assignments. The significance level is $0.05$, and $M$ is $10^5$ for the FRT.

\begin{example}[Balanced Experiments with $N_1=N_0$]
\label{eg::balanced}
The potential outcomes are independently
generated from Normal distributions $Y_i(1)\sim N(1/10, 1/16)$ and $Y_i(0)\sim N(0, 1/16)$, for $i=1, \ldots, 100$.
The individual causal effects are not constant, with $S_\tau^2=0.125.$
Further, once drawn from the Normal distributions above, they are fixed.
We repeatedly generate $1000$ completely randomized treatment assignments with $ N=100$ and $N_1  = N_0 = 50$. For each treatment assignment, we obtain the observed outcomes and implement two tests for Neyman's null and Fisher's null.
As shown in Table \ref{tb::test_b}, it never happens that we reject Fisher's null but fail to reject Neyman's null. However, 
we reject Neyman's null but fail to reject Fisher's null in $15$ instances.
\end{example}

\begin{example}[Unbalanced Experiments with $N_1\neq N_0$]
\label{eq::unbalanced}
The potential outcomes are independently
generated from Normal distributions $Y_i(1)\sim N(1/10, 1/4)$ and $Y_i(0)\sim N(0, 1/16)$, for $i=1,\ldots, 100$.
The individual causal effects are not constant, with $S_\tau^2=0.313.$
They are kept as fixed throughout the simulations. The unequal variances are designed on purpose, and
we will reveal the reason for choosing them later in Example \ref{eg::variance-compare} of Section \ref{sec::compare}.
We repeatedly generate $1000$ completely randomized treatment assignments with $ N=100, N_1  = 70,$ and $N_0 = 30$. 
After obtaining each observed data set, we perform two hypothesis testing procedures, and summarize the results in Table \ref{tb::test_unb}.
The pattern in Table \ref{tb::test_unb} is more striking than in Table \ref{tb::test_b}, because it happens $62$ times in Table \ref{tb::test_unb} that we reject Neyman's null but fail to reject Fisher's null.
For this particular set of potential outcomes, Neyman's testing procedure has a power $62/1000=0.062$, slightly larger than $0.05$, but Fisher's testing procedure has a power $8/1000=0.008$, much smaller than $0.05$ even though the sharp null is not true. We will explain in Section \ref{sec::compare} the reason why the FRT could have a power even smaller than the significance level under some alternative hypotheses.
\end{example}

\begin{table}[ht]
  \small
  \centering
  \caption{Numerical Examples.}
  \subtable[Balanced experiments with $N_1=N_0=50$, corresponding to Example \ref{eg::balanced}]{
\begin{tabular}{rcccc}
\hline 
\multicolumn{1}{r}{}
 &  \multicolumn{1}{c}{  $\text{ not reject }  H_0( \text{Fisher} )$} & \multicolumn{1}{c}{  $ \text{ reject }   H_0(\text{Fisher})$} \\
 \hline 
 $ \text{ not reject }  H_0( \text{Neyman} )$ & $488$  &  $0$ &  \\
 $ \text{ reject }  H_0( \text{Neyman} )$  &  $15$ & $497$ &power(Neyman)=0.512\\
 \hline 
 &&power(Fisher)=0.497
\end{tabular}\label{tb::test_b}
  }
  \\
  \subtable[Unbalanced experiments with $N_1=70$ and $N_0=30$, corresponding to Example \ref{eq::unbalanced}]{
\begin{tabular}{rcccc}
 \hline 
\multicolumn{1}{r}{}
 &  \multicolumn{1}{c}{  $ \text{ not reject } H_0( \text{Fisher} )$} & \multicolumn{1}{c}{  $ \text{ reject }   H_0(\text{Fisher})$}& \\
  \hline 
 $ \text{ not reject } H_0( \text{Neyman} )$ & $930$  &  $0$ &  \\
 $ \text{ reject }  H_0( \text{Neyman} )$  &  $62$ & $8$& power(Neyman)=0.070\\
  \hline 
  &&power(Fisher)=0.008
\end{tabular}\label{tb::test_unb}
  }
\end{table}

\subsection{Statistical Inference, Logic, and Paradox}

Logically, Fisher's null implies Neyman's null. Therefore, Fisher's null should be rejected if Neyman's null is rejected. However, this is not always true from the results of statistical inference in completely randomized experiments. 
We observed in our numerical examples above that it can be the case that
\begin{eqnarray}
p( \text{Neyman} )  < \alpha_0 < p( \text{Fisher} ) ,  \label{eq::controversy}
\end{eqnarray}
in which case we should reject Neyman's null, but not Fisher's null, if we choose the significance level to be $\alpha_0$ (e.g., $\alpha_0 = 0.05$).
When (\ref{eq::controversy}) holds, an awkward logical problem appears.
In the remaining part of this section, we will theoretically explain the empirical findings in Section \ref{sec::numerical-examples} and the consequential logical problem.

\subsection{Asymptotic Evaluations}
\label{sec::normal}

While Neyman's testing procedure has an explicit form, the FRT is typically approximated by Monte Carlo.
In order to compare them, we first discuss the asymptotic Normalities of $\widehat{\tau}$ and the randomization test statistic $\widehat{\tau}(\bm{T}, \bm{Y}^{\obs})$.
We provide a simplified way of doing variance calculation and a short proof for
asymptotic Normalities of both $\widehat{\tau}$ and $\widehat{\tau}(\bm{T}, \bm{Y}^{\obs})$, based on the finite population Central Limit Theorem \citep[CLT;][]{hoeffding::1952, hajek::1960, lehmann::1998, freedman::2008}.
Before the formal asymptotic results, it is worth mentioning the exact meaning of ``asymptotics'' in the context of finite population causal inference. We need to embed the finite population of interest into a hypothetical infinite sequence of finite populations with increasing sizes, and also require the proportions of the treatment units to converge to a fixed value. Essentially, all the population quantities (e.g., $\tau, S_1^2$, etc.) should have the index $N$, and all the sample quantities (e.g., $\widehat{\tau}, s_1^2$, etc.) should have double indices $N$ and $N_1$. However, for the purpose of notational simplicity, we sacrifice a little bit of mathematical precision and drop all the indices in our discussion.

\begin{theorem}
\label{thm::CLTneyman}
As $N\rightarrow \infty,$ the sampling distribution of $\widehat{\tau}$ satisfies
$$
\frac{ \widehat{ \tau } - \tau  }{ \sqrt{  \var( \widehat{ \tau } )  } }  \convD \mathcal{N}(0, 1).
$$
\end{theorem}

In practice, the true variance $\var(\widehat{ \tau })$ is replaced by its ``conservative'' estimator $\widehat{V}(\text{Neyman})$,
 and the resulting test rejects less often than the nominal significance level on average.
While the asymptotics for the Neymanian unbiased estimator $\widehat{\tau}$ does not depend on the null hypothesis, the following asymptotic Normality for 
$\widehat{\tau}(\bm{T}, \bm{Y}^{\obs})$ is true only under the sharp null hypothesis.

\begin{theorem}
\label{thm::CLTfisher} Under $H_0( \text{Fisher} ) $ and as $N\rightarrow \infty,$ the null distribution of $\widehat{\tau}(\bm{T},  \bm{Y}^{\obs}) $ satisfies
$$
\frac{ \widehat{\tau}(\bm{T},  \bm{Y}^{\obs})      }{ \sqrt{ \widehat{V}(\text{Fisher})}  } \convD  \mathcal{N}(0, 1),
$$
where $
\bar{Y}^{\obs} =   \sumN Y_i^{\obs} / N
$,
$ 
s^2 =  \sumN (Y_i^{\obs} - \bar{Y}^{\obs})^2 / (N-1),$
and
$
\widehat{V}(\text{Fisher})=  N s^2  / ( N_1 N_0 ) .$
\end{theorem}

Therefore, the $p$-value under $H_0( \text{Fisher} ) $ can be approximated by 
\begin{eqnarray}
p ( \text{Fisher})  \approx 2  \Phi \left\{  - \frac{  |\widehat{\tau}^{\obs}| }{  \sqrt{ \widehat{V}(\text{Fisher}) }   }  \right\} .
\label{eq::pfisher}
\end{eqnarray}
From (\ref{eq::pneyman}) and (\ref{eq::pfisher}), the asymptotic $p$-values obtained from Neymanian and Fisherian approaches differ only due to the difference between the variance estimators $\widehat{V}(\text{Neyman})$ and $\widehat{V}(\text{Fisher})$.
Therefore, a comparison of the variance estimators will explain the different behaviors of the corresponding approaches.
In the following, we use the conventional notation $R_N = o_p(N^{-1})$ for a random quantity satisfying $N \cdot R_N\rightarrow 0$ in probability, as $N\rightarrow \infty$ \citep[cf.][]{lehmann::1998}.

\begin{theorem}
\label{thm::diff}
Asymptotically, the difference between the two variance estimators is
\begin{eqnarray}
\widehat{V}(\text{Fisher}) - \widehat{V}(\text{Neyman})  
= (N_0^{-1} - N_1^{-1}) ( S^2_1  - S^2_0   ) 
+ N^{-1}(\bar{Y}_1 - \bar{Y}_0)^2 + o_p(N^{-1}) .\label{eq::diff}
\end{eqnarray}
\end{theorem}

The difference between the variance estimators depends on the ratio of the treatment and control sample sizes, and differences between the means and variances of the treatment and control potential outcomes.
The ``conservativeness'' of Neyman's test does not cause the paradox; if we use the true sampling variance rather than the estimated variance of $\widehat{\tau}$ for testing, then the paradox will happen even more often.

In order the verify the asymptotic theory above, we go back to compare the variances in the previous numerical examples.
\begin{example}[Continuations of Examples \ref{eg::balanced} and \ref{eq::unbalanced}]
\label{eg::variance-compare}
We plot in Figure \ref{fg::var} the variances $\widehat{V}(\text{Neyman})$ and $\widehat{V}(\text{Fisher})$ obtained from the numerical examples in Section \ref{sec::numerical-examples}. In both the left and the right panels, $\widehat{V}(\text{Fisher})$ tends to be larger than $\widehat{V}(\text{Neyman})$. This pattern is more striking on the right panel with unbalanced experiments designed to satisfy $ (N_0^{-1} - N_1^{-1}) \left( S^2_1  - S^2_0  \right)  > 0$. 
It is thus not very surprising that the FRT is much less powerful than Neyman's test, and it rejects even less often than nominal $0.05$ level as shown in Table \ref{tb::test_unb}.
\end{example}

\begin{figure}[ht]
\centering
\includegraphics[width =  0.8\textwidth]{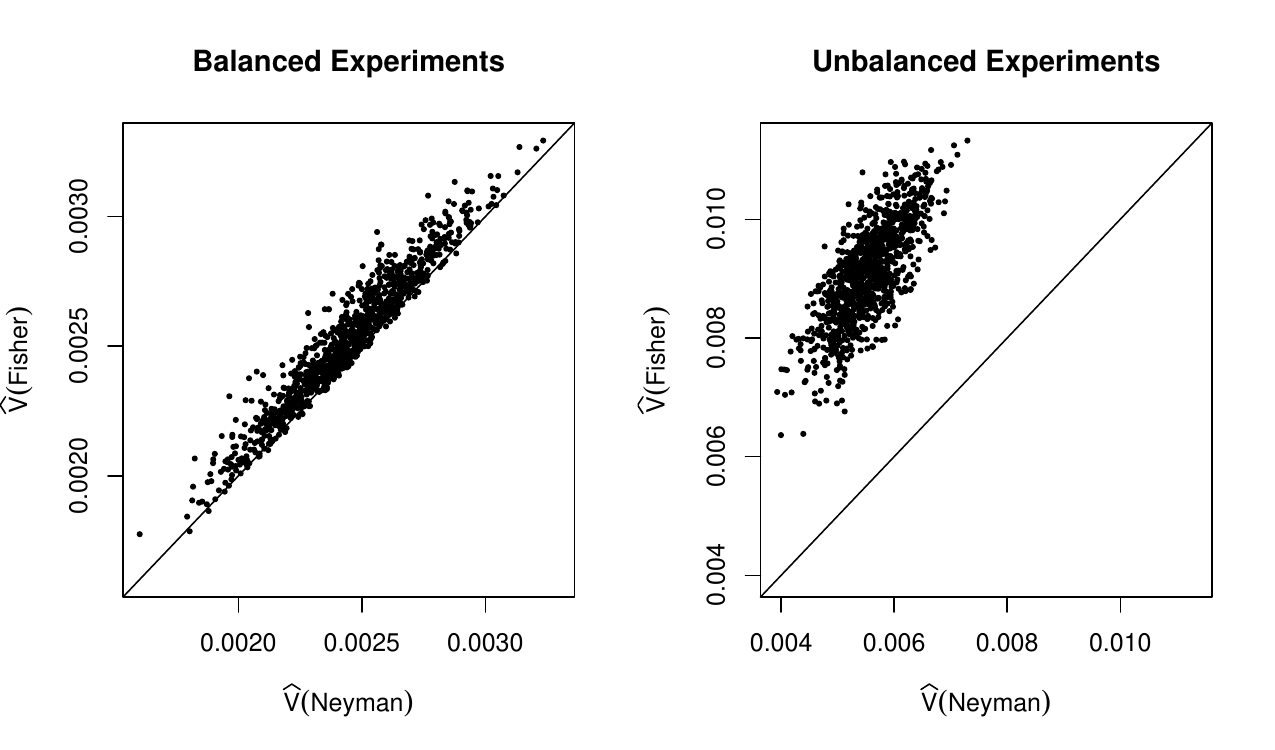}
\caption{Variance estimators in balanced and unbalanced experiments}\label{fg::var}
\end{figure}

\subsection{Theoretical Comparison}
\label{sec::compare}

Although quite straightforward, Theorem \ref{thm::diff} has several helpful implications to explain the paradoxical results in Section \ref{sec::numerical-examples}.

Under $H_0( \text{Fisher} ) $, $\bar{Y}_1 = \bar{Y}_0, S_1^2 = S_0^2$, and the difference between the two variances is of higher order, namely, $  \widehat{V}(\text{Fisher}) -\widehat{V}(\text{Neyman})  = o_p(N^{-1})$. Therefore, Neymanian and Fisherian methods coincide with each other asymptotically under the sharp null.
This is the basic requirement, because both testing procedures should generate correct type one errors under this circumstance.

For the case with constant causal effect, we have $\tau_i = \tau$ and $S_1^2 = S_0^2$. The difference between the two variance estimators reduces to
\begin{eqnarray}\label{eq::compare_balance}
 \widehat{V}(\text{Fisher}) - \widehat{V}(\text{Neyman}) 
= \tau^2/ N  + o_p(N^{-1}). 
\end{eqnarray}
Under $H_0(\text{Neyman})$, $\bar{Y}_1 = \bar{Y}_0$, and the difference between the two variances is of higher order, and two tests have the same asymptotic performance. However, under the alternative hypothesis, $\tau = \bar{Y}_1 - \bar{Y}_0 \neq 0$, and
the difference above is positive and of order $1/N$, and Neyman's test will reject more often than Fisher's test. With larger effect size $|\tau|$, the powers differ more.

For balanced experiments with
$N_1= N_0,$ the difference between the two variance estimators reduces to the same formula as (\ref{eq::compare_balance}), and the conclusions are the same as above.

For unbalanced experiments, the difference between two variances can be either positive or negative. 
In practice, if we have prior knowledge $S_1^2 > S_0^2$, unbalanced experiments with $N_1 > N_0$ are preferable to improve estimation precision. In this case, we have $(N_0^{-1} - N_1^{-1}) \left( S^2_1  - S^2_0  \right)  > 0$ and $\widehat{V}(\text{Fisher})  > \widehat{V}(\text{Neyman})$ for large $N$. Surprisingly, we are more likely to reject Neyman's null than Fisher's null, although Neyman's test itself is conservative with nonconstant causal effect implied by $S_1^2>S_0^2$.

From the above cases, we can see that Neymanian and Fisherian approaches generally have different performances, unless the sharp null hypothesis holds. Fisher's sharp null imposes more restrictions on the potential outcomes, and the variance of the randomization distribution of $\widehat{\tau}$ pools the within and between group variances across treatment and control arms. 
Consequently, the resulting randomization distribution of $\widehat{\tau}$ has larger variance than its repeated sampling variance in many realistic cases. Paradoxically, in many situations, we tend to reject Neyman's null more often than Fisher's null, which contradicts the logical fact that Fisher's null implies Neyman's null.

Finally, we consider the performance of the FRT under Neyman's null with $\bar{Y}_1 = \bar{Y}_0$, which is often of more interest in social sciences. 
If $ S_1^2>S_0^2$ and $N_1>N_0$, the rejection rate of Fisher's test is smaller than Neyman's test, even though $H_0(\text{Neyman})$ holds but $H_0(\text{Fisher})$ does not. Consequently, the difference-in-means statistic $\widehat{\tau}(\bm{T}, \bm{Y}^{\obs})$ has no power against the sharp null, and the resulting FRT rejects even less often than the nominal significance level. However, if $ S_1^2>S_0^2$ and $N_1<N_0$, the FRT may not be more ``conservative'' than Neyman's test. Unfortunately, the FRT may reject more often than the nominal level, yielding an invalid test for Neyman's null. \citet{gail::1996} and \citet{lang::2015} found this phenomenon in numerical examples, and we provide a theoretical explanation.

\subsection{Binary Outcomes}
\label{sec::binary-outcomes}

We close this section by investigating the special case with binary outcomes, for which more explicit results are available.
Let $p_t= \bar{Y}(t)$ be the potential proportion and $\widehat{p}_t = \bar{Y}_t^{\obs}$ be the sample proportion of one under treatment $t$. 
Define $\widehat{p} = \bar{Y}^{\obs}$ as the proportion of one in all the observed outcomes.
The results in the following 
corollary are special cases of Theorems \ref{thm::CLTneyman} to \ref{thm::diff}.

\begin{corollary}
\label{coro::binary}
Neyman's test is asymptotically equivalent to the ``unpooled'' test 
\begin{eqnarray}
\frac{ \widehat{p}_1 - \widehat{p}_0  }{   \sqrt{  \widehat{p}_1(1-\widehat{p}_1)  / N_1  + \widehat{p}_0 (1-\widehat{p}_0) / N_0   }     } 
\convD \mathcal{N}(0, 1)
\label{eq::unpooled}
\end{eqnarray}
under $H_0(\text{Neyman})$;
and Fisher's test is asymptotically equivalent to the ``pooled'' test
\begin{eqnarray}
\frac{   \widehat{p}_1 - \widehat{p}_0   }{  \sqrt{   \widehat{p}  (1-\widehat{p})  (N_1^{-1}  + N_0^{-1})     }  }
\convD \mathcal{N} (0,1)
\label{eq::pooled}
\end{eqnarray}
under $H_0(\text{Fisher})$.
The asymptotic difference between the two tests is due to
\begin{eqnarray}
&&\widehat{V}(\text{Fisher}) - \widehat{V}(\text{Neyman}) \nonumber \\
&=&
(N_0^{-1} - N_1^{-1})  \{  p_1 (1-p_1)   - p_0 (1  - p_0) \}  + N^{-1} (p_1 - p_0)^2 + o_p(N^{-1}).
\label{eq::compare-proportions}
\end{eqnarray}
\end{corollary}

For the case with binary outcomes, we can draw analogous but slightly different conclusions to the above.
Under Neyman's null, $p_1=p_0$ and the two tests are asymptotically equivalent. Therefore, the situation that the FRT is invalid under Neyman's null will never happen for binary outcomes.  
In balanced experiments, Neyman's test is always more powerful than Fisher's test under the alternative with $p_1\neq p_0.$
For unbalanced experiments, the answer is not definite, but Equation (\ref{eq::compare-proportions}) allows us to determine the region of $(p_1, p_0)$ that favors Neyman's test for a given level of the ratio $r=N_1/N.$ When $r>1/2$, Figure \ref{fg::binary-region} shows the regions in which Neyman's test is asymptotically more powerful than Fisher's test according to the value of $r$. When $r<1/2$, the region has the same shape by symmetry. We provide more details about Figure \ref{fg::binary-region} in the Supplementary Material.

\begin{figure}[ht]
\centering
\includegraphics[width = 0.8\textwidth]{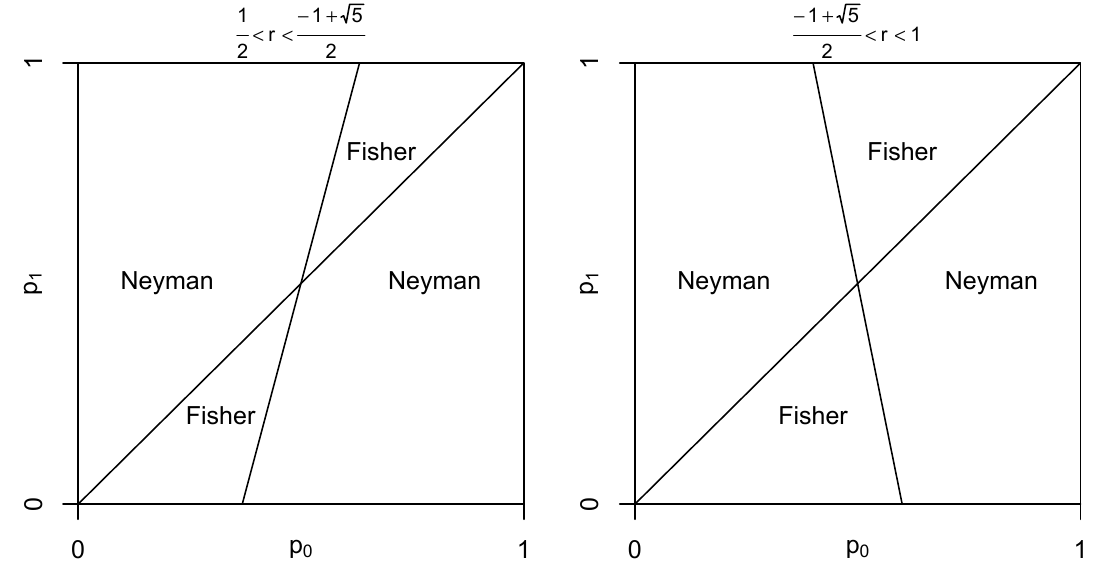}
\caption{Binary Outcome with Different Proportions $r=N_1/N$. Neyman's test is more powerful in the regions marked by ``Neyman.'' }\label{fg::binary-region}
\end{figure}

Note that Fisher's test is equivalent to Fisher's exact test, and (\ref{eq::pooled}) is essentially the Normal approximation of the hypergeometric distribution \citep{barnard::1947, cox::1970, ding::2015}. The two tests in (\ref{eq::unpooled}) and (\ref{eq::pooled}) are based purely on randomization inference, which have the same mathematical forms as the classical ``unpooled'' and ``pooled'' tests for equal proportions under two independent Binomial models. Our conclusion is coherent with \citet{robbins::1977} and \citet{eberhardt::1977} that the ``unpooled'' test is more powerful than the ``pooled'' one with equal sample size. For hypothesis testings in two by two tables, \citet{greenland::1991} observed similar theoretical results as Corollary \ref{coro::binary} but gave a different interpretation. 
Recently, \citet{rigdon::2015} and \citet{li::2016} constructed exact confidence intervals for $\tau$ by inverting a sequence of FRTs.

\section{Ubiquity of the Paradox in Other Experiments}
\label{sec::other-experiments}

The paradox discussed in Section \ref{sec::paradox-neyman-fisher} is not unique to completely randomized experiments.
As a direct generalization of the previous results, the paradox will appear in each stratum of stratified experiments.
We will also show its existence in two other widely-used experiments: matched-pair designs and factorial designs.
In order to minimize the confusion about the notation, each of the following two subsections is self-contained.

\subsection{Matched-Pair Experiments}

Consider a matched-pair experiment with $2N$ units and $N$ pairs matched according to their observed characteristics. Within each matched pair, we randomly select one unit to receive treatment and the other to receive control.
Let $T_i$ be iid Bernoulli$(1/2)$ for $i=1, \ldots, N$, indicating treatment assignments for the matched pairs. 
For pair $i$, the first unit receives treatment and the second unit receives control if $T_i=1$; and otherwise if $T_i=0.$
Under the SUTVA, we define $(Y_{ij}(1), Y_{ij}(0))$ as the potential outcomes of the $j$th unit in the $i$th pair under treatment and control, and the observed outcomes within pair $i$ are $Y_{i1}^{\obs} = T_i Y_{i1}(1) + (1-T_i) Y_{i1}(0)$ and $Y_{i2}^{\obs} = T_i Y_{i2}(0) + (1-T_i) Y_{i2} (1)$.
Let $\bm{T} = (T_1, \ldots, T_N) ' $ and $\bm{Y}^{\obs} = \{ Y_{ij}^{\obs} : i=1, \ldots, N; j=1,2\}$ denote the $N\times 1$ treatment assignment vector and the $N\times 2$ observed outcome matrix, respectively.
Within pair $i$, 
$$
\widehat{\tau}_i = T_i ( Y_{i1}^{\obs} - Y_{i2}^{\obs}  ) + (1-T_i) (Y_{i2}^{\obs}  - Y_{i1}^{\obs})
$$ 
is unbiased for the within-pair average causal effect 
$$
\tau_i = \{   Y_{i1}(1) + Y_{i2}(1) - Y_{i1}(0) - Y_{i2}(0)  \}/2.
$$
Immediately, we can use 
$$
\widehat{\tau} = \frac{1}{N} \sum_{i=1}^N \widehat{\tau}_i
$$
as an unbiased estimator for the finite population average causal effect
$$
\tau = \frac{1}{N} \sum_{i=1}^N \tau_i  = \frac{1}{2N} \sum_{i=1}^N \sum_{j=1}^2 \{  Y_{ij}(1)  - Y_{ij}(0) \}.
$$

\citet{imai::2008} discussed Neymanian inference for $\tau$ and identified the variance of $\widehat{\tau}$ with the corresponding variance estimator.
To be more specific, he calculated
$$
\var(\widehat{\tau}) = \frac{1}{4N^2} \sum_{i=1}^N  \{    Y_{i1}(1) + Y_{i1}(0) - Y_{i2}(1) - Y_{i2}(0)     \}^2,
$$
and proposed a variance estimator
$$
\widehat{V}( \text{Neyman} )  = \frac{1}{N(N-1)} \sum_{i=1}^N ( \widehat{\tau}_i  - \widehat{\tau} )^2.
$$
Again, the variance estimator is ``conservative'' for the true sampling variance because $E\{ \widehat{V}( \text{Neyman} ) \} \geq \var(\widehat{\tau}) $ unless the within-pair average causal effects are constant.
The repeated sampling evaluation above allows us to test Neyman's null hypothesis of zero average causal effect: 
$$
H_0(\text{Neyman}) : \tau = 0.
$$

On the other hand,
\citet{rosenbaum::2002} discussed intensively the FRT in matched-pair experiments under the sharp null hypothesis:
$$
H_0(\text{Fisher}) : Y_{ij}(1) = Y_{ij}(0)  , \forall i=1, \ldots, N; \forall j=1,2,
$$ 
which is, again, much stronger than Neyman's null.
For the purpose of comparison, we choose the test statistic with the same form as $\widehat{\tau}$, denoted as $\widehat{\tau}(\bm{T}, \bm{Y}^{\obs})  $. In fact, \citet{fisher::1935a} used this test to analyze Charles Darwin's data on the relative growth rates of cross- and self-fertilized corn. 
In practice, the null distribution of this test statistic can be calculated exactly by enumerating all the $2^N$ randomizations or approximated by Monte Carlo. For our theoretical investigation, we have the following results.

\begin{theorem}
\label{thm::match-pair}
Under the sharp null hypothesis, $E\{ \widehat{\tau}(\bm{T}, \bm{Y}^{\obs})\mid H_0(\text{Fisher})  \} = 0$, and 
$$
\widehat{V}( \text{Fisher} )  \equiv 
\var \{ \widehat{\tau}(\bm{T}, \bm{Y}^{\obs})\mid H_0(\text{Fisher})  \} =  
\frac{1}{N^2} \sum_{i=1}^N \widehat{\tau}_i^2.
$$ 
Therefore, for matched-pair experiments, the difference in the variances is
$$
\widehat{V}( \text{Fisher} ) -  \widehat{V}( \text{Neyman} )   =    \tau^2  / N   + o_p(N^{-1}) .
$$ 
\end{theorem}

The asymptotic Normality of the two test statistics holds because of the Lindberg--Feller CLT for independent random variables, and therefore the different power behaviors of Neyman and Fisher's tests is again due to the above difference in the variances.
Under $H_0(\text{Neyman})$, the difference is a higher order term, leading to asymptotically equivalent behaviors of Neymanian and Fisherian inferences. However, under the alternative hypothesis with nonzero $\tau$, the same paradox appears again in matched-pair experiments: we tend to reject with Neyman's test more often than with Fisher's test.

For matched-pair experiments with binary outcomes, we let $m_{y_1y_0}^{\obs} $ be the number of pairs with treatment outcome $y_1$ and control outcome $y_0$, where $y_1, y_0\in \{ 0, 1\}.$ Consequently, we can summarize the observed data by a two by two table with cell counts $(m_{11}^{\obs}, m_{10}^{\obs}, m_{01}^{\obs}, m_{00}^{\obs})$. Theorem \ref{thm::match-pair} can then be further simplified as follows.

\begin{corollary}
\label{coro::matched-pair}
In matched-pair experiments with binary outcomes, Neyman's test is asymptotically equivalent to
\begin{eqnarray}
\label{eq::matched-pair-neyman}
\frac{  m_{10}^{\obs} - m_{01}^{\obs} }
{  \sqrt{    m_{10}^{\obs} + m_{01}^{\obs}  -    (m_{10}^{\obs} - m_{01}^{\obs})^2 / N    }}
\convD \mathcal{N}(0, 1)
\end{eqnarray}
under $H_0$(Neyman), and Fisher's test is asymptotically equivalent to
\begin{eqnarray}
\label{eq::matched-pair-fisher}
\frac{  m_{10}^{\obs} - m_{01}^{\obs} } { \sqrt{m_{10}^{\obs} + m_{01}^{\obs}} }
\convD  \mathcal{N}(0, 1) 
\end{eqnarray}
under $H_0$(Fisher). And the asymptotic difference between the two tests is due to
$$
\widehat{V}( \text{Fisher} ) -  \widehat{V}( \text{Neyman} )   =    (  m_{10}^{\obs} - m_{01}^{\obs}  )^2  / N^3   + o_p(N^{-1}) .
$$
\end{corollary}

Note that the number of discordant pairs, $  m_{10}^{\obs} + m_{01}^{\obs}$, is fixed over all randomizations under the sharp null hypothesis, and therefore Fisher's test is equivalent to the exact test based on $m_{10}^{\obs}\sim$ Binomial$( m_{10}^{\obs} + m_{01}^{\obs} , 1/2 )$. Its asymptotic form (\ref{eq::matched-pair-fisher}) is the same as the McNemar test under a super population model \citep{agresti::2004}.

\subsection{Factorial Experiments}
\label{sec::factorial}

\citet{fisher::1935a} and \citet{yates::1937} developed the classical factorial experiments in the context of agricultural experiments, and \cite{wu::2009} provided a comprehensive modern discussion of design and analysis of factorial experiments. 
Although rooted in randomization theory \citep{kempthorne::1955, hinkelmann::2007}, the analysis of factorial experiments is dominated by linear and generalized linear models, with factorial effects often defined as model parameters.
Realizing the inherent drawbacks of the predominant approaches, \cite{dasgupta::2012} discussed causal inference from $2^K$ factorial experiments using the potential outcomes framework, which allows for defining the causal estimands based on potential outcomes instead of model parameters.

We first briefly review the notation for factorial experiments adopted by  
\cite{dasgupta::2012}.
Assume that we have $K$ factors with levels $+1$ and $-1$. Let $\bm{z} = (z_1, \ldots, z_K)' \in \mathcal{F}_K =  \{ +1, -1\}^K$, a $K$-dimensional vector, denote a particular treatment combination.
The number of possible values of $\bm{z}$ is $J = 2^K$, for each of which we define $Y_i(\bm{z})$ as the corresponding potential outcome for unit $i$ under the SUTVA.
We use a $J$-dimensional vector $\bm{Y}_i$ to denote all potential outcomes for unit $i$, where $i=1, \ldots, N =  r\times 2^K$ with an integer $r$ representing the number of replications of each treatment combination.  
Without loss of generality, we will discuss the inference of the main factorial effect of factor $1$, 
and analogous discussion also holds for general factorial effects due to symmetry.
The main factorial effect of factor $1$ can be characterized by a vector $\bm{g}_1$ of dimension $J$, with one half of its elements being $+1$ and the other half being $-1.$
Specifically, the element of $\bm{g}_1$ is $+1$ if the corresponding $z_1$ is $+1$, and $-1$ otherwise.
For example, in $2^2$ experiments, we have $\bm{Y}_i = (Y_i(+1,+1), Y_i(+1,-1), Y_i(-1,+1), Y_i(-1,-1))'$ and $\bm{g}_1 = (+1,+1,-1,-1)'.$
We define $\tau_{i1} = 2^{-(K-1)}  \bm{g}_1' \bm{Y}_i$ as the main factorial effect of factor $1$ for unit $i$, and
\begin{eqnarray*}
\tau_1 = \frac{1}{N} \sum_{i=1}^N \tau_{i 1}  = 2^{-(K-1)} \bm{g}_1' \bar { \bm{Y} }
\end{eqnarray*}
as the average main factorial effect of the factor $1$,
where $\bar{ \bm{Y}} = \sum_{i=1}^N \bm{Y}_i / N.$

For factorial experiments, we define the treatment assignment as $W_i(\bm{z})$, with $W_i(\bm{z})  = 1$ if the $i$th unit is assigned to $\bm{z}$, and $0$ otherwise.
Therefore, we use $\bm{W}_i = \{ W_i(\bm{z}) : \bm{z}  \in \mathcal{F}_K  \}$ as the treatment assignment vector for unit $i$, and let $\bm{W}$ be the collection of all the unit-level treatment assignments.
The observed outcomes are deterministic functions of the potential outcomes and the treatment assignment, namely, $Y_i^{\obs} = \sum_{\bm{z} \in \mathcal{F}_K}  W_i(\bm{z})  Y_i(\bm{z})$ for unit $i$, and $\bm{Y}^{\obs} = (Y_1^{\obs}, \ldots, Y_N^{\obs}) ' $ for all the observed outcomes. 
Because
\begin{eqnarray*}
\bar{Y}^{\obs}(\bm{z}) = \frac{1}{r} \sum_{ \{ i: W_i(\bm{z} )   = 1 \} }  Y_i^{\obs} = \frac{1}{r} \sum_{i=1}^N  W_i(\bm{z}) Y_i(\bm{z})
\end{eqnarray*}
is unbiased for $\bar{Y} (\bm{z})$, 
we can unbiasedly estimate $\tau_1$ by
\begin{eqnarray*}
\widehat{\tau}_1  = 2^{-(K-1)} \bm{g}_1' \bar{ \bm{Y}}^{\obs} ,
\end{eqnarray*}
where $\bar{ \bm{Y}}^{\obs}$ is the $J$-dimensional vector for the average observed outcomes.
\cite{dasgupta::2012} showed that the sampling variance of $\widehat{\tau}_1$ is
\begin{eqnarray}
\label{eq::factorial-var}
\var(\widehat{\tau}_1) = \frac{1}{2^{2(K-1)} r} \sum_{\bm{z} \in \mathcal{F}_K }  S^2(\bm{z})  - \frac{1}{N} S_1^2,
\end{eqnarray}
where $S^2(\bm{z}) = \sum_{i=1}^N   \{   Y_i(\bm{z}) - \bar{Y}(\bm{z})   \}^2/(N-1) $ is the finite population variance of the potential outcomes under treatment combination $\bm{z}$, and $S_1^2 = \sum_{i=1}^N   (\tau_{i1} - \tau_1)^2/(N-1)$ is the finite population variance of the unit level factorial effects $\{ \tau_{i1} :  i=1, \ldots, N\}$. Similar to the discussion in completely randomized experiments, the last term $S_1^2$ in (\ref{eq::factorial-var}) cannot be identified, and consequently the variance in (\ref{eq::factorial-var}) can only be ``conservatively'' estimated by the following Neyman-style variance estimator:
\begin{eqnarray*}
\widehat{V}_1(\text{Neyman}) = \frac{1}{ 2^{2(K-1)}  r}  \sum_{\bm{z} \in \mathcal{F}_K } s^2(\bm{z}),
\end{eqnarray*}
where the sample variance of outcomes $s^2(\bm{z}) = \sum_{ \{ i: W_i (\bm{z}) = 1 \} }  \{    Y_i^{\obs} - \bar{Y}^{\obs}(\bm{z})   \}^2/(r-1)$ under treatment combination $\bm{z}$ is unbiased for $S^2(\bm{z})$.
The discussion above allows us to construct a Wald-type test for Neyman's null of zero average factorial effect for factor $1$:
$$
H_{0}^1 (\text{Neyman}): \tau_1 = 0.
$$

On the other hand, based on the physical act of randomization in factorial experiments, the FRT allows us to test the following sharp null hypothesis:
\begin{eqnarray}
H_{0} (\text{Fisher}) : Y_i(\bm{z}) = Y_i^{\obs}, \forall \bm{z}\in \mathcal{F}_K, \forall  i=1,\ldots, N.
\label{eq::factorial-fisher-null}
\end{eqnarray}
This sharp null restricts all factorial effects for all the individuals to be zero,
which is much stronger than $H_{0}^1 (\text{Neyman}).$
For a fair comparison, we use the same test statistic as $\widehat{\tau}_1  $ in our randomization test, and denote $\widehat{\tau}_1 ( \bm{W}, \bm{Y}^{\obs} ) $ as a function of the treatment assignment and observed outcomes.
Under the sharp null (\ref{eq::factorial-fisher-null}), the randomness of $\widehat{\tau}_1 ( \bm{W}, \bm{Y}^{\obs} ) $ is induced by randomization, and the following theorem gives us its mean and variance.

\begin{theorem}
\label{thm::factorial-fisher-randomization}
Under the sharp null,
$
E\{   \widehat{\tau}_1 ( \bm{W}, \bm{Y}^{\obs} )   \mid    H_{0} (\text{Fisher}) \} = 0,
$
and
$$
\widehat{V}_1 (\text{Fisher})  \equiv    \var\{   \widehat{\tau}_1 ( \bm{W}, \bm{Y}^{\obs} )   \mid    H_{0} (\text{Fisher}) \}= \frac{1}{2^{2(K-1)}  r} Js^2,
$$
where 
$
\bar{Y}^{\obs} =  \sum_{i=1}^N Y_i^{\obs}/N $ and $ s^2 = \sum_{i=1}^N  (  Y_i^{\obs}  - \bar{Y}^{\obs}  )^2/(N-1) 
$
are the sample mean and variance of all the observed outcomes.
\end{theorem}

Based on Normal approximations, comparison of the $p$-values reduces to the difference between $\widehat{V}_1(\text{Neyman}) $ and $\widehat{V}_1 (\text{Fisher}) $, as shown in the theorem below.

\begin{theorem}
\label{thm::factorial-logical-incoherence}
With large $r$,
the difference between $\widehat{V}_1(\text{Neyman}) $ and $\widehat{V}_1 (\text{Fisher}) $ is
\begin{eqnarray}
\widehat{V}_1(\text{Fisher}) - \widehat{V}_1(\text{Neyman})  
= \frac{1}{2^{3K-1} r}  \sum_{\bm{z}\in \mathcal{F}_K}    \sum_{\bm{z} ' \in \mathcal{F}_K} 
\{ \bar{Y}(\bm{z})  -   \bar{Y}(\bm{z} ' )  \}^2 + o_p(r^{-1}) .\label{eq::variance-difference-factorial-experiment}
\end{eqnarray}
\end{theorem}

Formula (\ref{eq::compare_balance}) is a special case of formula (\ref{eq::variance-difference-factorial-experiment}) with $K=1$ and $r=N_1=N_0 = N/2$, because complete randomized experiments are special cases of factorial experiments with a single factor.
Therefore, in factorial experiments with the same replicates $r$ at each level, the paradox always exists under alternative hypothesis with nonzero $\tau_1$, just as in balanced completely randomized experiments.

\section{Improvements and Extensions}
\label{sec::extensions}

We have shown that the seemingly paradoxical phenomenon in Section \ref{sec::paradox-neyman-fisher} is due to the fact that Neyman's test is more powerful than Fisher's test in many realistic situations. 
The previous sections restrict the discussion on the difference-in-means statistic. We will further comment on the importance of this choice, and other possible alternative test statistics.
Moreover, the original forms of Neyman's and Fisher's tests are both suboptimal. We will discuss improved Neymanian and Fisherian inference, and the corresponding paradox.

\subsection{Choice of the Test Statistic}
\label{subsec::test-stat}

First, as hinted by \citet{ding::2015}, for randomized experiments with binary outcomes, all test statistics are equivalent to the difference-in-means statistic. We formally state this conclusion in the following theorem.

\begin{theorem}
\label{thm::equivalent-difference-in-means}
For completely randomized experiments, matched-pair experiments, and $2^K$ factorial experiments, if the outcomes are binary, then all test statistics are equivalent to the difference-in-means statistic. 
\end{theorem}

Therefore, for binary data, the choice of test statistic is not a problem.

Second, for continuous outcomes, the difference-in-means statistic is important, because it not only serves as a candidate test statistic for the sharp null hypothesis but also an unbiased estimator for the average causal effect. In the illustrating example in Section \ref{sec::factorial-example}, practitioners are interested in finding the combination of several factors that achieves an optimal mean response.

For continuous outcomes we have more options of test statistics. For instance, the Kolmogorov--Smirnov and Wilcoxon--Mann--Whitney statistics are also useful candidates for the FRT. However, the Neymanian analogues of these two statistics have not been established in the literature, and direct comparisons of the Fisherian and Neymanian using these two statistics are not obvious at this moment. In the Supplementary Material, we illustrate by numerical examples that the conservative nature of the FRT is likely to be true for these two statistics, because we find that the randomization distributions under the sharp null hypothesis is more disperse than those under weaker null hypotheses. Please see the Supplementary Materials for more details, and it is our future research topic to pursue the theoretical results.

\subsection{Improving the Neymanian Variance Estimators}

For completely randomized experiments, \citet{neyman::1923} used $S_\tau^2\geq 0$ as a lower bound, which is not the sharp bound. Recently, for general outcomes \citet{aronow::2014} derived the sharp bound of $S_\tau^2$ based on the marginal distributions of the treatment and control potential outcomes using the Frech\'et--Hoeffding bounds \citep{nelsen::2007}; for binary outcomes \citet{robins::1988} and \citet{ding::2015} gave simple forms. These improvements result in smaller variance estimators.

For matched-pair experiments, \citet{imai::2008} improved the Neymanian variance estimator by using the Cauchy--Schwarz inquality. We are currently working on deriving sharp bounds for the variance of estimated factorial effects.

In summary, Neyman's test is even more powerful with improved variance estimators, which further bolsters the paradoxical situation wherein we reject Neyman's null but fail to reject Fisher's sharp null.

\subsection{Improving the FRT and Connection With the Permutation Test}

In the permutation test literature, some authors \citep[e.g.,][]{neuhaus::1993,janssen::1997,chung::2013,pauly::2015} suggested using the Studentized version of $\widehat{\tau}$, i.e., $\widehat{\tau}/ \widehat{V}_{\text{Neyman}}^{1/2}$, as the test statistic. When the experiment is unbalanced, the FRT using this test statistic has exact type one error under Fisher' null and correct asymptotic type one error under Neyman's null. However, this does not eliminate the paradox discussed in this paper. 
First, we have shown that this paradox arises even in balanced experiments, but this test statistic tries to correct the invalid asymptotic type one error under Neyman's null in unbalanced experiments. Second, Section \ref{subsec::test-stat} has shown that for binary outcomes any test statistic is equivalent to $\widehat{\tau}$, and therefore this Studentized test statistic will not change the paradox at least for binary outcomes. Third, the theories of permutation tests and randomization tests do not have a one-to-one mapping, although they often give the same numerical results. The theory of permutation tests assumes exchangeable units drawn from an infinite super-population, and the theory of randomization tests assumes fixed potential outcomes in a finite population and random treatment assignment. Consequently, the correlation between the potential outcomes never plays a role in the theory of permutation tests, but it plays a central role in the theory of randomization inference as indicated by \citet{neyman::1923}'s seminal work and our discussion above.

\section{Illustrations}
\label{sec::examples}

In this section, we will use real-life examples to illustrate the theory in the previous sections. The first two examples have binary outcomes, and therefore there is no concern about the choice of test statistic. The goal of the third example, a $2^4$ full factorial experiment, is to find the optimal combination of the factors, and therefore the difference-in-means statistic is again a natural choice for a test statistic.

\subsection{A Completely Randomized Experiment}

Consider a hypothetical completely randomized experiment with binary outcome \citep[][pp.191]{rosenbaum::2002}. Among the $32$ treated units, $18$ of them have outcome being $1$, and among the $21$ control units, $5$ of them have outcome being $1$. The Neymanian $p$-value based on the improved variance estimator in \citet{robins::1988} and \citet{ding::2015} is $0.004$. The Fisherian $p$-value based on the FRT or equivalently Fisher's exact test is $0.026$, and the Fisherian $p$-value based on Normal approximation in (\ref{eq::pooled}) is $0.020$. The Neymanian $p$-value is smaller, and if we choose significance level at $0.01$ then the paradox will appear in this example.

\subsection{A Matched-Pair Experiment}

The observed data of the matched-pair experiment in \citet{agresti::2004} can be summarized by the two by two table with cell counts $(m_{11}^{\obs}, m_{10}^{\obs},m_{01}^{\obs},m_{00}^{\obs}) = (53,8,16,9)$. The Neymanian one-sided $p$-value based on (\ref{eq::matched-pair-neyman}) is $0.049$. The Fisherian $p$-value based on the FRT is $0.076$, and the Fisherian $p$-value based on Normal approximation in (\ref{eq::matched-pair-fisher}) is $0.051.$ Again, Neyman's test is more powerful than Fisher's test.

\subsection{A $2^4$ Full Factorial Experiment}
\label{sec::factorial-example}

In the ``Design of Experiments'' course in Fall 2014, a group of Harvard undergraduate students, Taylor Garden, Jessica Izhakoff and Zoe Rosenthal, followed \citet{box::1992}'s famous paper helicopter example for factorial experiments, and tried to identify the optimal combination of the four factors: paper type (construction paper, printer paper), paperclip type (small paperclip, large paperclip), wing length ($2.5$ inches, $2.25$ inches), and fold length ($0.5$ inch, $1.0$ inch), with the first level coded as $-1$ and the second level coded as $+1.$ For more details, please see \citet{box::1992}. For each combination of the factors, they recorded two replicates of the flying times of the helicopters. We display the data in Table \ref{Stat140data}.

We show the Neymanian and Fisherian results in the upper and lower panel of Figure \ref{fg::factorial}, respectively. Figure \ref{fg::neyman} shows both Neymanian point estimates and $p$-values for the $15$ factorial effects. Seven of them, $F_1, F_2, F_4, F_1F_2, F_1F_3, F_1F_4$ and $F_1F_2F_4$, are significant at level $0.05$, and after the Bonferroni correction, three of them, $F_1,F_2,F_1F_2F_4$, are still significant. Figure \ref{fg::fisher} shows the randomization distribution of the factorial effects under the sharp null hypothesis by a grey histogram. Note that all factorial effects have the same randomization distribution, because all of them are essentially a comparison of a random half versus the other half of the observed outcomes. Even though the sample size $32$ is not huge, the randomization distribution is well approximated by the Normal distribution with mean zero and variance $\widehat{V}_1(\text{Fisher})$. Strikingly, only two factorial effects, $F_1$ and $F_2$, are significant, and after the Bonferroni correction only $F_2$ is significant. We further calculate the variance estimates: $\widehat{V}_1(\text{Neyman}) = 0.025$ and $\widehat{V}_1(\text{Fisher})=0.034$. The empirical findings in this particular example with finite sample are coherent with our asymptotic theory developed in Section \ref{sec::factorial}.
In this example, the Neymanian method can help detect more significant factors for achieving optimal flying time, while the more conservative Fisherian method may miss important factors.

\begin{table}[ht]
\centering
\caption{A $2^4$ Factorial Design and Observed Outcomes} 
\begin{tabular}{|rrrr|cc|cc|}
  \hline
$F_1$ & $F_2$ & $F_3$ & $F_4$ & replicate 1 & replicate 2   \\ 
  \hline
$-1$ & $-1$ & $-1$ & $-1$ & 1.60 & 1.55   \\ 
  $-1$ & $-1$ & $-1$ & 1 & 1.70 & 1.63   \\ 
  $-1$ & $-1$ & 1 & $-1$ & 1.44 & 1.38   \\ 
  $-1$ & $-1$ & 1 & 1 & 1.56 & 1.61   \\ 
  $-1$ & 1 & $-1$ & $-1$ & 1.40 & 1.45   \\ 
  $-1$ & 1 & $-1$ & 1 & 1.36 & 1.38   \\ 
  $-1$ & 1 & 1 & $-1$ & 1.43 & 1.40   \\ 
  $-1$ & 1 & 1 & 1 & 1.32 & 1.27   \\ 
  1 & $-1$ & $-1$ & $-1$ & 1.81 & 1.86   \\ 
  1 & $-1$ & $-1$ & 1 & 1.70 & 1.57  \\ 
  1 & $-1$ & 1 & $-1$ & 2.04 & 2.06   \\ 
  1 & $-1$ & 1 & 1 & 1.68 & 1.61   \\ 
  1 & 1 & $-1$ & $-1$ & 1.58 & 1.28  \\ 
  1 & 1 & $-1$ & 1 & 1.43 & 1.49   \\ 
  1 & 1 & 1 & $-1$ & 1.51 & 1.54   \\ 
  1 & 1 & 1 & 1 & 1.53 & 1.38   \\ 
   \hline
\end{tabular}
\label{Stat140data}
\end{table}

\begin{figure}[!ht]
\centering
\subfigure[Neymanian Inference. Factorial effects $F_1, F_2, F_4, F_1F_2, F_1F_3, F_1F_4$ and $F_1F_2F_4$ are significant at level $0.05$.]{
    \includegraphics[width = \textwidth]{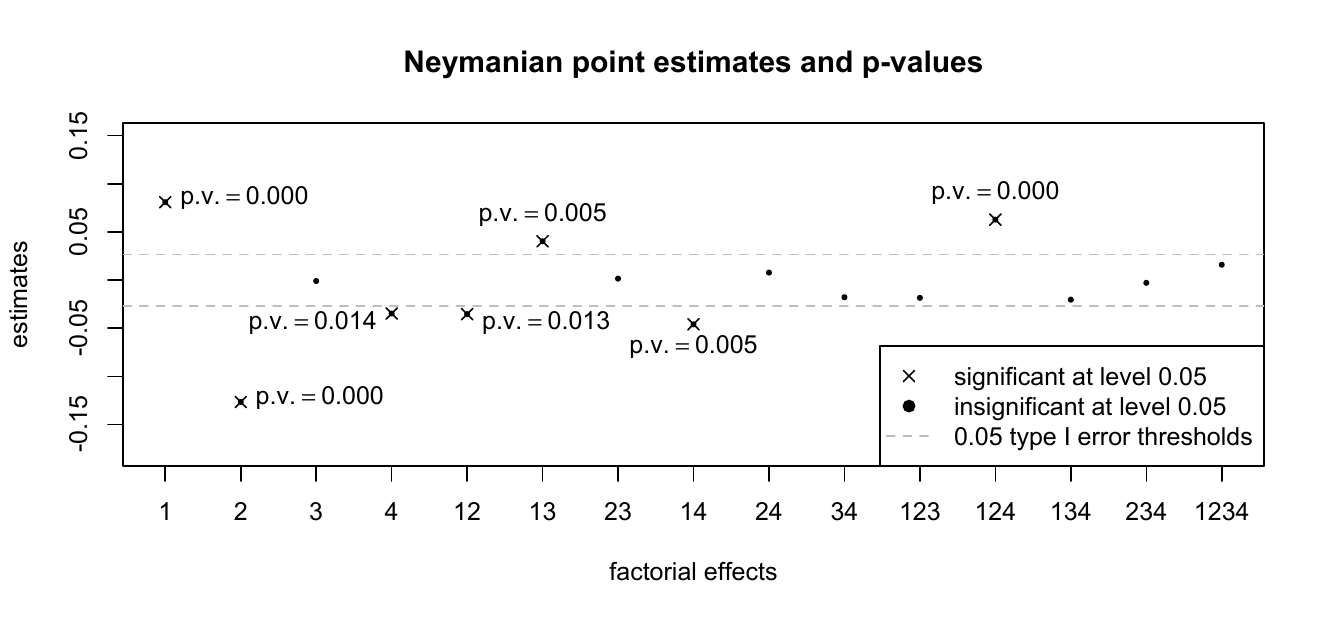}
    \label{fg::neyman}
}
\subfigure[Fisherian Inference. Factorial effects $F_1$ and $F_2$ are significant.]{
    \includegraphics[width = \textwidth]{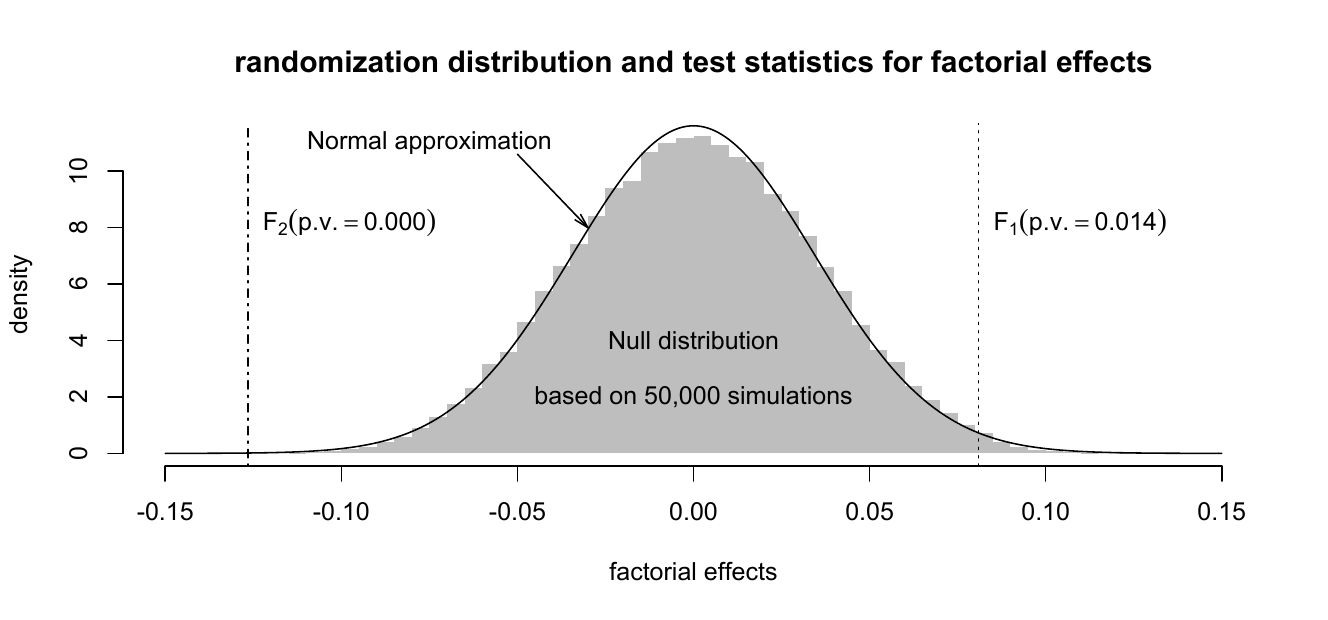}
    \label{fg::fisher}
}
\caption{Randomization-Based Inference for a $2^4$ Full Factorial Experiment}
\label{fg::factorial}
\end{figure}

\section{Discussion}
\label{sec::discussion}

\subsection{Historical Controversy and Modern Discussion}

\citet{neyman::1923} proposed to use potential outcomes for causal inference and derived mathematical properties of randomization; \citet{fisher::1926} advocated using randomization in physical experiments, which was considered by Neyman ``as one of the most valuable of Fisher's achievements'' \citep[][page 44]{reid::1982}. 
\citet[][Section II]{fisher::1935a} pointed out that ``the actual and physical conduct of an experiment must govern the statistical procedure of its interpretation.'' Neyman and Fisher both proposed statistical procedures for analysis of randomized experiments, relying on the randomization distribution itself. However,
whether Neyman's null or Fisher's null makes more sense in practice goes back to the famous Neyman--Fisher controversy in a meeting of the Royal Statistical Society \citep{neyman::1935, fisher::1935b}. 
After their 1935 controversy, \citet{anscombe::1948}, \citet{kempthorne::1952} and \citet{cox::1958} provided some further discussion on the usefulness and limitations of the two null hypotheses. For instance, the authors acknowledged that Neyman's null is mathematically weaker than Fisher's null, but both null hypotheses seem artificial requiring either individual causal effects or the average causal effect be exactly zero for finite experimental units.  For Latin square designs, \citet{wilk::1957} developed theory under Neyman's view, and \citet{cox::1958latin} argued that in most situations the Fisherian analysis was secure.  
Recently, \citet[page 39]{rosenbaum::2002} gave a very insightful philosophical discussion about the controversy, and \cite{sabbaghi::2013} revisited this controversy and its consequences. \citet{fienberg::1996} and \citet{cox::2012} provided more historical aspects of causal inference and in particular the Neyman--Fisher controversy.

While the answer may depend on different perspectives of practical problems, we discussed only the consequent seeming paradox of Neymanian and Fisherian testing procedures for their own null hypotheses. 
Both our numerical examples and asymptotic theory showed that we encounter a serious logical problem in the analysis of randomized experiments, even though both Neyman's and Fisher's tests are valid Frequentists' tests, in the sense of controlling correct type one errors under their own null hypotheses.
Our numerical examples and theoretical analysis reach a conclusion different from
 \citet{rosenbaum::2002}.

\subsection{Randomization-Based and Regression-Based Inference}

In the current statistical practice, it is also very popular among applied researchers to use regression-based methods to analyze experimental data \citep{angrist::2008}. Assume the 
a linear model for the observed outcomes:
$
Y_i^{\obs} = \alpha + \beta T_i + \varepsilon_i, 
$
where $\varepsilon_i, \ldots, \varepsilon_N$ are independently and identically distributed (iid) as $\mathcal{N}(0,\sigma^2)$.
The hypothesis of zero treatment effect is thus characterized by
$
H_0(LM) :  \beta = 0 . 
$
The usual ordinary least squares variance estimator for the regression coefficient may not correctly reflect the true variance of $\widehat{\tau}$ under randomization. \citet{schochet::2010},
\citet{samii::2012}, \citet{lin::2013} and \citet{imbens::2015} pointed out that we can solve this problem by using the Huber--White heteroskedasticity-robust variance estimator \citep{huber::1967, white::1980}, and the corresponding Wald test is asymptotically the same as Neyman's test. In Theorem A.1 of the Supplementary Material, we further build an equivalence relationship between Rao's score test and the FRT. For more technical details, please see the Supplementary Material. Previous results, as well as Theorem A.1, do justify the usage of linear models in analysis of experimental data.

\subsection{Interval Estimation}
Originally, \citet{neyman::1923} proposed an unbiased estimator for the average causal effect $\tau$ with a repeated sampling evaluation, which was later developed into the concept of the confidence interval \citep{neyman::1937}. In order to compare Neyman's approach with the FRT, we converted the interval estimator into a hypothesis testing procedure. As a dual, we can also invert the 
FRT for a sequence of null hypotheses to get an interval estimator for $\tau$ \citep{pitman::1937, pitman::1938, rosenbaum::2002}.
For example, we consider the sequence of sharp null hypotheses with constant causal effects:
\begin{eqnarray} 
H_0^ \delta(\text{Fisher}): Y_i(1) -  Y_i(0) =  \delta , \quad  \forall i = 1, \ldots, N.
\label{eq::constantH0}
\end{eqnarray}
The interval estimator for $\tau$ with coverage rate $1-\alpha$ is
$$
\left\{   \delta :  \text{Fail to reject }H_0^ \delta(\text{Fisher}) \text{ by the FRT at significant level }\alpha  \right\}.
$$ 
\citet{dasgupta::2012} 
found some empirical evidence in factorial designs that the above interval is wider than the Neymanian confidence interval.
Due to the duality between hypothesis testing and interval estimation, our results about hypothesis testing can partially explain the phenomenon about interval estimation in \citet{dasgupta::2012}. To avoid making assumptions such as constant causal effects in \eqref{eq::constantH0}, we restricted the theoretic discussion to only hypothesis testings. It is our future work to extend the theory to interval estimations.

\subsection{Practical Implications}

We highlight some practical implications of our theory developed in the above sections.

First, the FRT is usually less powerful than Neyman's test, even for the simplest case with constant causal effect. Practitioners should keep in mind that the FRT may miss important treatment factors. Our examples in Section \ref{sec::examples} and the empirical evidence in \citet{dasgupta::2012} have confirmed our theoretical results.

Second, in the presence of treatment effect heterogeneity, the FRT may not be a valid test for the null hypothesis of zero average causal effect. Therefore, practitioners, especially those who are interested in social sciences, should always be aware of this potential danger of using the FRT, if the observed data show substantive heterogeneity in treatment and control groups. Furthermore, as \citet{cox::1958latin} pointed out, in the presence of treatment effect heterogeneity, focusing only on the average causal effect is often not adequate, and detecting and explaining such heterogeneity may be more helpful.
Treatment effect variation is another important issue beyond the current scope of our paper. \citet{ding::2015jrssb} investigate this problem under the randomization framework.

Third, although we have shown that the FRT is less powerful in many realistic cases, we do not conclude that Neymanian inference trumps Fisherian inference. All our comparisons are based on asymptotics under regularity conditions, and the conclusion may not be true with small sample sizes or ``irregular'' potential outcomes. Therefore, Fisherian inference is still useful for small sample problems and exact inference. In practice, we should always check the discrepancy between the Normal approximation and the exact randomization distribution as in Figure \ref{fg::fisher} before applying our theoretical results to applied problems.

\section*{Supplementary Material}

Appendix A.1 gives two useful lemmas for randomized experiments. Appendix A.2 gives the proofs of all the theorems and corollaries in the main text. Appendix A.3 comments on the regression-based causal inference, and establishes a new connection between Rao's score test and the FRT. Appendix A.4 shows more details about generating Figure \ref{fg::binary-region} in the main text. Appendix A.5 discusses the behaviors of the FRT using the Kolmogorov--Smirnov and Wilcoxon--Mann--Whitney statistics.

\bibliographystyle{Chicago}

\pagebreak 
\bigskip
\begin{center}
{\large\bf Supplementary Material}
\end{center}

\renewcommand {\thefigure} {A\arabic{figure}}
\renewcommand {\thetable} {A\arabic{table}}
\renewcommand {\theequation} {A\arabic{equation}}
\renewcommand {\thelemma} {A\arabic{lemma}}
\renewcommand {\thesection} {A\arabic{section}}
\renewcommand {\thetheorem} {A\arabic{theorem}}

\setcounter{section}{0}

\section{Lemmas}

\begin{lemma}
\label{lemma::randomization}
The completely randomized treatment assignment $\bm{T} = (T_1, \ldots, T_N) ' $ satisfies 
$$
E ( T_i )  = \frac{N_1}{ N}, \quad 
\var(T_i) = \frac{  N_1 N_0 }{N^2}, \quad 
\cov(T_i, T_j) = - \frac{N_1 N_0 } { N^2 (N-1)  }.
$$ 
If $c_1, \ldots, c_N$ are constants and $\bar{c} = \sumN c_i/N$, we have
\begin{eqnarray*}
E\left(   \sumN T_i c_i  \right)  = N_1 \bar{c}, \quad 
\var\left(  \sumN T_i c_i \right) = \frac{N_1 N_0}{N (N-1)} \sumN (c_i - \bar{c})^2.
\end{eqnarray*}
\end{lemma}

\begin{proof}
[Proof of Lemma \ref{lemma::randomization}.] 
The treatment vector $\bm{T}$ can be viewed as the inclusion indicator vector of a simple random sample of size $N_1$ from a finite population of size $N$. The conclusion follows from \citet{cochran::1977}.
\end{proof}

\begin{lemma}
[Finite Population Central Limit Theorem; \citealp{hajek::1960, lehmann::1998}]
\label{lemma::CLT}
Suppose we have a finite population $ \{ x_{1}, \ldots, x_{N} \}$ with size $N$ and mean $\bar{x} =  \sumN x_i / N$, and a simple random sample of size $n$ with inclusion indicators $\{I_i:i=1, \ldots, N\}$. Let $\bar{X}_n = \sumN I_i x_{i}/n$ be the sample mean. As $N\rightarrow \infty$, if 
\begin{eqnarray}
\label{eq::hajekcondition}
{  \max_{1\leq i\leq N}(x_i - \bar{x})^2 \over   \sumN (x_i - \bar{x})^2/N } \text{ is bounded and } { n \over N} \rightarrow c\in (0, 1),
\end{eqnarray}
we have that 
$$
\frac{  \bar{X}_n  - \bar{x}    }{  \sqrt{  \var ( \bar{X}_n ) } } \convD N(0,1).
$$
\end{lemma}

\begin{lemma}
\label{lemma::factorial}
If $\{  W_i(\bm{z}) : i=1,\ldots, N; \bm{z}\in \mathcal{F}_K \}$ is the collection of treatment indicators from a $2^K$ factorial experiment, then we have the following correlation structure: for $i\neq i'$ and $\bm{z}\neq \bm{z}'$,
\begin{eqnarray*}
\cov\{ W_i(\bm{z}), W_{i}(\bm{z})  \} = \frac{r(N-r)}{N^2},\quad 
\cov\{ W_i(\bm{z}), W_{i'}(\bm{z})  \} = - \frac{ r(N-r) }{  N^2(N-1) } ,\\
\cov\{ W_i(\bm{z}), W_{i}(\bm{z}')  \} = - \frac{r^2}{N^2},\quad 
\cov\{ W_i(\bm{z}), W_{i'}(\bm{z}')  \} = \frac{r^2}{ N^2(N-1)}. 
\end{eqnarray*}
\end{lemma}

\begin{proof}
[Proof of Lemma \ref{lemma::factorial}.]
See Lemmas 4 and 5 of
\citet{dasgupta::2012}. 
\end{proof}

\section{Proofs of the Theorems}

\begin{proof}[Proof of Theorem 1.]
First, $\widehat{\tau}$ has the following representation
\begin{eqnarray}
\widehat{\tau} &=& \frac{1}{N_1} \sumN T_i Y_i^{\obs} - \frac{1}{N_0} \sumN (1 - T_i)Y_i^{\obs}   \nonumber\\
                       &=& \frac{1}{N_1}  \sumN T_i Y_i(1) - \frac{1}{N_0} \sumN (1 - T_i)Y_i(0)  \nonumber \\
                       &=& \sumN T_i \left\{   \frac{Y_i(1)}{N_1} + \frac{Y_i(0)}{N_0}  \right\} - \frac{1}{N_0}\sumN Y_i(0). \label{eq::representNeyman}
\end{eqnarray}
Since all the potential outcomes are fixed, we use Lemma \ref{lemma::randomization} to obtain that the mean is
$$
E( \widehat{\tau}  ) = \frac{N_1}{N}  \sumN  \left\{   \frac{Y_i(1)}{N_1} + \frac{Y_i(0)}{N_0}  \right\} - \frac{1}{N_0}\sumN Y_i(0)
= \frac{1}{N} \sumN Y_i(1) - \frac{1}{N} \sumN Y_i(0)  = \tau, 
$$
and the variance is
\begin{eqnarray*}
\var(  \widehat{\tau} )  &=& \frac{N_1 N_0}{N (N-1)} \sumN \left\{    \frac{Y_i(1)}{N_1} + \frac{Y_i(0)}{N_0} -   \frac{\bar{Y}_1}{N_1} - \frac{\bar{Y}_0}{N_0}   \right\}^2\\
                                    &=& \frac{N_1 N_0}{N (N-1)} 
                                           \left[    \frac{1}{N_1^2}  \sumN \{ Y_i(1) - \bar{Y}_1\}^2 
                                                     + \frac{1}{N_0^2}  \sumN \{ Y_i(0) - \bar{Y}_0\}^2  
                                                     \right.\\
                                    &&  \left.                 
                                                     +   \frac{2}{N_1 N_0}  \sumN \{ Y_i(1) - \bar{Y}_1 \}\{Y_i(0) - \bar{Y}_0\}
                                           \right].
\end{eqnarray*}
Because of the following decomposition based on $2ab = a^2 + b^2 - (a-b)^2$:
$$2 \{ Y_i(1) - \bar{Y}_1 \}\{ Y_i(0) - \bar{Y}_0 \} = \{ Y_i(1) - \bar{Y}_1 \}^2 +  \{ Y_i(0) - \bar{Y}_ 0 \}^2 - \{ Y_i(1) - Y_i(0) - \bar{Y}_1   +\bar{Y}_0  \}^2 ,$$ 
we have 
$
2S_{10} = S_1^2 + S_0^2 - S_\tau^2,
$
and therefore
$$
\var(\widehat{\tau}) = \frac{  S_1^2} {N_1 } + \frac{ S_0^2 }{  N_0 } - \frac{ S_\tau^2} { N} .
$$
Furthermore, $\sumN T_i \left\{    Y_i(1) /N_1 + Y_i(0) / N_0\right\} / N_1 $ is the mean of a simple random sample from $\left\{  x_i =    Y_i(1) /N_1 + Y_i(0) / N_0   : i=1, \ldots, N \right\}$, and the asymptotic Normality of $\widehat{\tau}$ follows from (\ref{eq::representNeyman}) and Lemma \ref{lemma::CLT} if $ x_i =    Y_i(1) /N_1 + Y_i(0) / N_0  $ satisfies the condition in (\ref{eq::hajekcondition}).
\end{proof}

\begin{proof}[Proof of Theorem 2.]
Under Fisher's sharp null, all the potential outcomes are fixed constants with $Y_i(1) = Y_i(0) = Y_i^{\obs}$.
The randomization statistic can be represented as
\begin{eqnarray}
\widehat{\tau}(\bm{T}, \bm{Y}^{\obs})  &=& \frac{1}{N_1} \sumN T_i Y_i^{\obs} - \frac{1}{N_0} \sumN (1 - T_i)Y_i^{\obs}  \nonumber \\
                         &=&  \frac{N}{N_1 N_0} \sumN T_i Y_i^{\obs} - \frac{1}{N_0} \sumN Y_i^{\obs}  . \label{eq::representFisher}
\end{eqnarray}
Using Lemma \ref{lemma::randomization}, we have 
\begin{eqnarray*}
E\left\{ \widehat{\tau}(\bm{T}, \bm{Y}^{\obs})  \mid H_0( \text{Fisher} )   \right\}  
&=& \frac{N}{N_1 N_0} \frac{N_1}{N} \sumN  Y_i^{\obs} - \frac{1}{N_0} \sumN Y_i^{\obs} =0,\\
\var\left\{ \widehat{\tau}(\bm{T}, \bm{Y}^{\obs}) \mid H_0( \text{Fisher} )   \right\} 
&=& \frac{N}{N_1 N_0 (N-1)} \sumN (Y_i^{\obs} - \bar{Y}^{\obs})^2  .
\end{eqnarray*}
Because $  \sumN T_i Y_i^{\obs}  / N_1 $ is the mean of a simple random sample from $\left\{ x_i = Y_i^{\obs}:1, \ldots, N  \right\}$, the randomization statistic $\widehat{\tau}(\bm{T}, \bm{Y}^{\obs})  $ follows a Normal distribution asymptotically by (\ref{eq::representFisher}) and Lemma \ref{lemma::CLT} if $x_i = Y_i^{\obs}$ satisfies the condition in (\ref{eq::hajekcondition}).
\end{proof}

\begin{proof}[Proof of Theorem 3.]
We have the following variance decomposition for $\bm{Y}^{\obs}$:
\begin{eqnarray*}
&&\sumN (Y_i^{\obs} - \bar{Y}^{\obs})^2 \\
&=& \sum\limits_{ \{ i:  T_i=1\}  } (Y_i^{\obs} - \bar{Y}_1^{\obs} +  \bar{Y}_1^{\obs} - \bar{Y}^{\obs})^2 + \sum\limits_{ \{ i: T_i=0\} } (Y_i^{\obs} - \bar{Y}_0^{\obs} +  \bar{Y}_0^{\obs} - \bar{Y}^{\obs})^2 \\
&=& \sum\limits_{ \{ i: T_i=1\} }   (Y_i^{\obs} - \bar{Y}_1^{\obs})^2 + N_1  (  \bar{Y}_1^{\obs} - \bar{Y}^{\obs})^2 +  \sum\limits_{ \{ i: T_i=0\} }   (Y_i^{\obs} - \bar{Y}_0^{\obs})^2 + N_0  (  \bar{Y}_0^{\obs} - \bar{Y}^{\obs})^2.
\end{eqnarray*}

Ignoring the difference between $N$ and $N-1$ contributes only a higher order term $o_p(N^{-1})$ in the asymptotic analysis. Therefore, we obtain that
\begin{eqnarray*}
&&\widehat{V}(\text{Fisher}) - \widehat{V}(\text{Neyman}) \\
&=& N_0^{-1} s_1^2 + N_1^{-1} s_0^2   +   
N_0^{-1}  (  \bar{Y}_1^{\obs} - \bar{Y}^{\obs})^2 
 + N_1^{-1}  (  \bar{Y}_0^{\obs} - \bar{Y}^{\obs})^2  
- N_1^{-1} s_1^2 - N_0^{-1} s_0^2  + o_p(N^{-1}) \\
&=& (N_0^{-1} - N_1^{-1}) (s_1^2 - s_0^2) +  N_0^{-1}  (  \bar{Y}_1^{\obs} - \bar{Y}^{\obs})^2 + N_1^{-1}(  \bar{Y}_0^{\obs} - \bar{Y}^{\obs})^2   + o_p(N^{-1}).
\end{eqnarray*}
Since $\bar{Y}^{\obs} = (N_1 \bar{Y}_1^{\obs} + N_0 \bar{Y}_0^{\obs}) / N$, we have 
$$
(\bar{Y}_1^{\obs} - \bar{Y}^{\obs} )^2 / N_0 =  N_0  ( \bar{Y}_1^{\obs} - \bar{Y}_0^{\obs})^2 / N^2, \quad 
(\bar{Y}_0^{\obs} - \bar{Y}^{\obs})^2 / N_1  = N_1 ( \bar{Y}_1^{\obs} - \bar{Y}_0^{\obs} )^2/N^2.
$$
It follows that
$$
\widehat{V}(\text{Fisher}) - \widehat{V}(\text{Neyman})  = 
(N_0^{-1} - N_1^{-1}) (s_1^2 - s_0^2) + N^{-1}  ( \bar{Y}_1^{\obs} - \bar{Y}_0^{\obs} )^2 + o_p(N^{-1}).
$$
Replacing the sample quantities $(s_1^2, s_0^2, \bar{Y}_1^{\obs} , \bar{Y}_0^{\obs})$
by the population quantities $(S_1^2, S_0^2, \bar{Y}_1, \bar{Y}_0)$ adds only higher order terms $o_p(N^{-1})$, and we eventually have
$$
\widehat{V}(\text{Fisher}) - \widehat{V}(\text{Neyman})  = 
(N_0^{-1} - N_1^{-1}) (S_1^2 - S_0^2) + N^{-1}  ( \bar{Y}_1 - \bar{Y}_0)^2 + o_p(N^{-1}).
$$
\end{proof}

\begin{proof}[Proof of Corollary 1.]
For binary outcomes, the conclusions follow from
\begin{eqnarray*}
s_t^2 &=&  \frac{1}{N_t-1} \sum_{ \{i:T_i=t\} } (Y_i^{\obs} - \bar{Y}_t^{\obs})^2 = \frac{N_t}{N_t - 1} \widehat{p}_t (1-\widehat{p}_t) , \\ 
s^2 &=& \frac{1}{N-1}  \sum_{i=1}^N (Y_i^{\obs} - \bar{Y}^{\obs})^2 = \frac{N}{N-1} \widehat{p} (1-\widehat{p}) .
\end{eqnarray*}
\end{proof}

\begin{proof}
[Proof of Theorem 4.]
Under the sharp null hypothesis, $\{  |\widehat{\tau}_i| = | Y_{i1}^{\obs} - Y_{i2}^{\obs}  | : i=1, \ldots, N\} $ are all fixed numbers, and $\widehat{\tau}(\bm{T}, \bm{Y}^{\obs})$ has the same distribution as
$$
\widehat{\tau}(\bm{T}, \bm{Y}^{\obs})  \sim  \frac{1}{N} \sumN (1-2T_i) |\widehat{\tau}_i| 
\sim \frac{1}{N} \sumN \delta_i |\widehat{\tau}_i| ,
$$
where $\delta_i$'s are iid random signs with mean zero and variance one. Therefore, the randomization distribution of $\widehat{\tau}(\bm{T}, \bm{Y}^{\obs})$ has mean zero by symmetry, and variance
$$
\widehat{V}(\text{Fisher}) = 
\var\{ \widehat{\tau}(\bm{T}, \bm{Y}^{\obs}) \mid   H_0(\text{Fisher}) \} 
= \frac{1}{N^2} \sumN  \var(\delta_i)  | \widehat{\tau}_i | ^2 
= \frac{1}{N^2} \sumN    \widehat{\tau}_i ^2 .
$$
The classical Lindberg--Feller Central Limit Theorem \citep{lehmann::1998} guarantees its asymptotic normality. 

The difference between the Neymanian and Fisherian variances is
\begin{eqnarray*}
\widehat{V}(\text{Fisher})  - \widehat{V}(\text{Neyman})&=&
  \frac{1}{N^2} \sumN    \widehat{\tau}_i ^2 - \frac{1}{N(N-1)} \sumN  ( \widehat{\tau}_i - \widehat{\tau} )^2 \\
  &=&  \frac{1}{N^2} \sumN    \widehat{\tau}_i ^2 - 
    \frac{1}{N^2}  \left(  \sumN   \widehat{\tau}_i ^2  - N \widehat{\tau}^2 \right)  + o_p(N^{-1}) \\
    &=& \frac{\tau^2}{N}  + o_p(N^{-1}),
\end{eqnarray*}
where the $o_p(N^{-1})$ appears due to the difference between $N$ and $N-1$, and $\widehat{\tau} - \tau = o_p(1).$
\end{proof}

\begin{proof}
[Proof of Corollary 2.]
For matched-pair experiments with binary outcomes, we have
$$
\widehat{\tau} = \frac{1}{N} \sumN \widehat{\tau}_i = \frac{m_{10}^{\obs} - m_{01}^{\obs} } { N },
$$
since only the pairs with discordant outcomes contribute to the $\widehat{\tau}_i$ terms. 
The Fisherian variance is
$$
\widehat{V}(\text{Fisher}) =    \frac{1}{N^2} \sumN    \widehat{\tau}_i ^2
= \frac{ m_{10}^{\obs} + m_{01}^{\obs} }{N^2} ,
$$
and the Neymanian variance is
$$
\widehat{V}(\text{Neyman}) = 
\frac{1}{N(N-1)}  \left(  \sumN   \widehat{\tau}_i ^2  - N \widehat{\tau}^2 \right) 
= \frac{1}{N(N-1)}  \left\{      m_{10}^{\obs} + m_{01}^{\obs}   - \frac{  (m_{10}^{\obs} - m_{01}^{\obs})^2 } { N } \right\} .
$$
Therefore, the Fisherian test is asymptotically equivalent to
$$
\frac{\widehat{\tau} }{ \sqrt{   \widehat{V}(\text{Fisher})  } }
= \frac{    m_{10}^{\obs} - m_{01}^{\obs}     }{  \sqrt{ m_{10}^{\obs} + m_{01}^{\obs}  }}
\convD \mathcal{N}(0,1)
$$
under $H_0$(Fisher), and the Neymanian test is asymptotically equivalent to 
$$
\frac{\widehat{\tau} }{ \sqrt{   \widehat{V}(\text{Neyman})  } }
=  \frac{    m_{10}^{\obs} - m_{01}^{\obs}     }{  \sqrt{ m_{10}^{\obs} + m_{01}^{\obs}  - (m_{10}^{\obs} - m_{01}^{\obs})^2/N  }}
\convD \mathcal{N}(0,1)
$$
under $H_0$(Neyman).
\end{proof}

\begin{proof}[Proof of Theorem 5.]
It is direct to obtain $E\{   \widehat{\tau}_1 ( \bm{W}, \bm{Y}^{\obs} )   \mid    H_{0} (\text{Fisher}) \} = 0$ by symmetry.
Under $H_{0} (\text{Fisher}) $, $ \bm{Y}^{\obs} =  \{ Y_i^{\obs}:i=1,\ldots,N \}$ is a fixed vector. Lemma \ref{lemma::factorial} implies that $ \bar{Y}^{\obs} (\bm{z})$ is the sample mean of a simple random sample of size $r$ from the population $ \bm{Y}^{\obs}  $ of size $N$. Therefore, we have
\begin{eqnarray}
\var\{   \bar{Y}^{\obs} (\bm{z}) \mid H_{0} (\text{Fisher})   \}  = \left(  \frac{1}{r} - \frac{1}{N} \right) s^2.  \label{eq::lemma-factorial1}
\end{eqnarray}  
Based on the correlation structure in Lemma \ref{lemma::factorial}, we obtain that
\begin{eqnarray}
&&\cov\{   \bar{Y}^{\obs} (\bm{z}_1) ,  \bar{Y}^{\obs} (\bm{z}_2)  \mid H_{0} (\text{Fisher})  \}  \nonumber \\
&=& 
\frac{1}{r^2} \cov\left\{  \sumN W_i(\bm{z}_1) Y_i^{\obs},  \sumN W_i(\bm{z}_2) Y_i^{\obs}   \mid H_{0} (\text{Fisher})   \right\}  \nonumber  \\
&=& \frac{1}{r^2}  \left[   \sumN \cov\{   W_i(\bm{z}_1), W_i(\bm{z}_2) \}   (Y_i - \bar{Y}^{\obs})^2  \right. \nonumber \\
&&~~~~~~~~~ \left. + \sumN \sum_{i'\neq i}  \cov\{   W_i(\bm{z}_1), W_i(\bm{z}_2) \} (Y_i - \bar{Y}^{\obs}) (Y_{i'} - \bar{Y}^{\obs}) \right]   \nonumber  \\
&=&  -  \frac{1}{N^2} \sumN   (Y_i - \bar{Y}^{\obs})^2   + \frac{1}{N^2(N-1)}   \sumN \sum_{i'\neq i} (Y_i - \bar{Y}^{\obs}) (Y_{i'} - \bar{Y}^{\obs})  \nonumber  \\
&=& -  \frac{1}{N^2} \sumN   (Y_i - \bar{Y}^{\obs})^2  -  \frac{1}{N^2(N-1)}  \sumN   (Y_i - \bar{Y}^{\obs})^2  \nonumber  \\
&=& -  \frac{1}{N} s^2. \label{eq::lemma-factorial2}
\end{eqnarray}  
Therefore, the variance of the test statistic is
\begin{eqnarray*}
&&\var\{   \widehat{\tau}_1 ( \bm{W}, \bm{Y}^{\obs} )   \mid    H_{0} (\text{Fisher}) \} \\
&=& 2^{ -2(K-1) } \bm{g}_1 '  \cov(  \bar{\bm{Y}}^{\obs} ) \bm{g}_1\\
&=& 2^{ -2(K-1) } \left[   \sum_{j=1}^J  g_{1j}^2 \var\{   \bar{Y}^{\obs} (\bm{z}_j) \mid H_{0} (\text{Fisher})    \}   \right. \nonumber \\
&&\left.~~~~~~~~~~~~~~~ +  \sum_{j=1}^J \sum_{j'\neq j}^J g_{1j} g_{1j'}   \cov\{   \bar{Y}^{\obs} (\bm{z}_j) ,  \bar{Y}^{\obs} (\bm{z}_{j'})  \mid H_{0} (\text{Fisher})  \}  \right] \\
&=& 2^{ -2(K-1) }  s^2  \left\{   \sum_{j=1}^J  g_{1j}^2 \left( \frac{1}{r} - \frac{1}{N} \right) -     \sum_{j=1}^J \sum_{j'\neq j}^J  g_{1j}g_{1j'}   \frac{1}{N}   \right\} ,
\end{eqnarray*}
where the last equation is due to (\ref{eq::lemma-factorial1}) and (\ref{eq::lemma-factorial2}).
Since 
$$
0 =\left(  \sum_{j=1}^J g_{1j}  \right)^2 = \sum_{j=1}^J g_{1j}^2 +   \sum_{j=1}^J \sum_{j'\neq j}^J  g_{1j}g_{1j'},
$$ 
we have 
$$ 
-  \sum_{j=1}^J \sum_{j'\neq j}^J  g_{1j}g_{1j'} = \sum_{j=1}^J g_{1j}^2 = J.
$$ 
Therefore, we can simplify the variance as
\begin{eqnarray*}
&&\var\{   \widehat{\tau}_1 ( \bm{W}, \bm{Y}^{\obs} )   \mid    H_{0} (\text{Fisher}) \} 
= 2^{ -2(K-1) }  s^2 J/r.
\end{eqnarray*}
\end{proof}

\begin{proof}[Proof of Theorem 6.]
We first observe the following variance decomposition:
\begin{eqnarray*}
&&\sum_{i=1}^N  (  Y_i^{\obs} - \bar{Y}^{\obs} )^2 \\
&=& \sum_{\bm{z}\in \mathcal{F}_K} \sum_{\{ i: W_i(\bm{z}) = 1 \} } 
\{    Y_i^{\obs} -   \bar{Y}^{\obs}(\bm{z})  +   \bar{Y}^{\obs}(\bm{z})  - \bar{Y}^{\obs} \}^2\\
&=&\sum_{\bm{z}\in \mathcal{F}_K} \sum_{\{ i: W_i(\bm{z}) = 1 \} } 
\{    Y_i^{\obs} -   \bar{Y}^{\obs}(\bm{z}) \}^2  + r \sum_{\bm{z}\in \mathcal{F}_K}   \{ \bar{Y}^{\obs}(\bm{z})  - \bar{Y}^{\obs} \}^2 .
\end{eqnarray*}
Therefore, we have
\begin{eqnarray*}
s^2 &=& \frac{1}{N-1} \sum_{\bm{z}\in \mathcal{F}_K} \sum_{\{ i: W_i(\bm{z}) = 1 \} } 
\{    Y_i^{\obs} -   \bar{Y}^{\obs}(\bm{z}) \}^2  +\frac{r}{N-1}   \sum_{\bm{z}\in \mathcal{F}_K}   \{ \bar{Y}^{\obs}(\bm{z})  - \bar{Y}^{\obs} \}^2\\
&=&  \frac{r-1}{N-1} \sum_{\bm{z}\in \mathcal{F}_K} s^2(\bm{z}) + \frac{r}{N-1} \sum_{\bm{z}\in \mathcal{F}_K}   \{ \bar{Y}^{\obs}(\bm{z})  - \bar{Y}^{\obs} \}^2\\
& = & \frac{1}{J} \sum_{\bm{z}\in \mathcal{F}_K} s^2(\bm{z}) + \frac{1}{J} \sum_{\bm{z}\in \mathcal{F}_K}   \{ \bar{Y}^{\obs}(\bm{z})  - \bar{Y}^{\obs} \}^2 + o_p(r^{-1}),
\end{eqnarray*}
where ignoring the difference between $N$ and $N-1$ and between $r$ and $r-1$ in the last equation contributes the higher order term.
Therefore, we have
\begin{eqnarray*}
2^{2(K-1)} r \left\{ \widehat{V}_1(\text{Fisher}) - \widehat{V}_1(\text{Neyman}) \right\} 
= Js^2 -  \sum_{\bm{z}\in \mathcal{F}_K} s^2(\bm{z}) 
 =  \sum_{\bm{z}\in \mathcal{F}_K}   \{ \bar{Y}^{\obs}(\bm{z})  - \bar{Y}^{\obs} \}^2 + o_p(r^{-1}) .
\end{eqnarray*}
Since 
$
\bar{Y}^{\obs} =  \sum_{\bm{z}\in \mathcal{F}_K}  \bar{Y}^{\obs}(\bm{z}) / 2^K$, the formula $\sum_{i=1}^{n} ( x_i  - \bar{x} )^2 = \sum_{i = 1}^n \sum_{j=1}^n (x_i - x_j)^2/(2n) $ gives us
$$
\sum_{\bm{z}\in \mathcal{F}_K}   \{ \bar{Y}^{\obs}(\bm{z})  - \bar{Y}^{\obs} \}^2
= \sum_{\bm{z}\in \mathcal{F}_K} \sum_{\bm{z} ' \in \mathcal{F}_K} \{ \bar{Y}^{\obs}(\bm{z})  - \bar{Y}^{\obs}(\bm{z}')  \}^2/2^{K+1}.
$$
Consequently, we have
$$
\widehat{V}_1(\text{Fisher}) - \widehat{V}_1(\text{Neyman}) 
= \frac{1}{2^{3K-1} r}  \sum_{\bm{z}\in \mathcal{F}_K}    \sum_{\bm{z} ' \in \mathcal{F}_K} 
\{ \bar{Y}^{\obs}(\bm{z})  -   \bar{Y}^{\obs}(\bm{z'})  \}^2 + o_p(r^{-1}) ,
$$
which leads to the final conclusion since replacing $\bar{Y}^{\obs}(\bm{z})$ by $\bar{Y}(\bm{z})$ contributes only $o_p(r^{-1})$.
\end{proof}

\begin{proof}
[Proof of Theorem 7.]
In the following, we will prove the results for completely randomized experiments, matched-pair experiments, and factorial experiments, respectively.

For completely randomized experiments with binary outcomes, we can summarize the observed data by a two by two table with cell counts $n_{ty}^{\obs} = \#\{ i: T_i=t, Y_i^{\obs} = y \}$, where $t,y=0,1$. The row sums $N_1 = n_{11}^{\obs} + n_{10}^{\obs}$ and $N_0 = n_{01}^{\obs} + n_{00}^{\obs}$ are fixed by the design of experiments, and the column sums $ n_{11}^{\obs} + n_{01}^{\obs}$ and $n_{10}^{\obs} + n_{00}^{\obs}$ are also fixed under the sharp null hypothesis. Therefore, $n_{11}^{\obs}$ is the only random component in the two by two table, because other cell counts are deterministic functions of it. According to the treatment assignment mechanism, we know that $n_{11}^{\obs}$ follows the hypergeometric distribution the same as the one in Fisher's exact test. All test statistics are functions of the two by two table, and thus functions of $n_{11}^{\obs}$. Consequently, all test statistics are equivalent to the difference-in-means statistic under the sharp null.

For matched-pair experiments with binary outcomes, we can summarize the observed data by the two by two table with cell counts $m_{y_1y_0}^{\obs}$ defined in the main text. Under the sharp null hypothesis, $m_{11}^{\obs}$, $m_{00}^{\obs}$, and $m_{dis}^{\obs} =  m_{10}^{\obs}+m_{01}^{\obs}$ are all fixed numbers, implying that the only random component in the two by two table is $m_{10}^{\obs}.$ According to the treatment assignment mechanism, we know $m_{10}^{\obs} \sim$ Binomial$(m_{dis}^{\obs} , 1/2 )$. All test statistics are functions of the two by two table, and thus functions of $m_{10}^{\obs}$. Consequently, all test statistics are equivalent to the difference-in-means statistic under the sharp null.

For $2^K$ factorial experiments, by symmetry we only need to show the result for factorial effect 1. It has the same structure as completely randomized experiments, and therefore, the conclusion follows.
\end{proof}

\section{Connections with Regression-Based Inference}
\label{sec::regression-inference}

 Assume the 
following linear model for the observed outcomes:
\begin{eqnarray}
\label{eq::LM}
Y_i^{\obs} = \alpha + \beta T_i + \varepsilon_i, 
\end{eqnarray}
where $\varepsilon_i, \ldots, \varepsilon_N$ are independently and identically distributed (iid) as $\mathcal{N}(0,\sigma^2)$.
The hypothesis of zero treatment effect is thus characterized by
$
H_0(LM) :  \beta = 0 . 
$

\cite{hinkelmann::2007} called 
$$
Y_{i}^{\obs} = T_iY_i(1) + (1 - T_i)Y_i(0) = Y_i(0) +\{ Y_i(1) - Y_i(0) \}T_i = \alpha + \beta T_i +\varepsilon_i
$$ 
the ``derived linear model'', assuming that $Y_i(1) - Y_i(0) = \beta$ is a constant and $Y_i(0) = \alpha + \varepsilon_i$ for all $i=1,\ldots, N.$ 
But the linear model for observed outcomes ignores the design of the randomized experiment, and the ``iid'' assumption contradicts $\text{cov}(T_i, T_j)\neq 0$ and $\text{cov}(Y_i^{\obs}, Y_j^{\obs}) \neq 0$ for $i\neq j$.
Although linear regression has been criticized for analyzing experimental data \citep{freedman::2008}, the least square estimator $\widehat{\beta}_{OLS}  = \widehat{\tau}$ is unbiased for the average causal effect $\tau$. 
However, the correct variance of $\widehat{\beta}_{OLS} $ requires careful discussion.

\subsection{Wald Test and Neymanian Inference}
\label{sec::huber-white}

The residual is defined as $\widehat{\varepsilon}_i = Y_i^{\obs} - \bar{Y}_1$ if $T_i = 1$ and $\widehat{\varepsilon}_i = Y_i^{\obs} - \bar{Y}_0$ if $T_i = 0$. Since the variance $\sigma^2$ in the linear model can be estimated by 
\begin{eqnarray*}
\widehat{\sigma}^2 =  \frac{1}{N-2} \sumN \widehat{\varepsilon}_i^2 
=  \frac{N_1 - 1}{N-2} s_1^2 + \frac{N_0 - 1}{N-2} s_0^2,
\end{eqnarray*}
the variance of $\widehat{\beta}_{OLS}$,
$
\Var (  \widehat{\beta}_{OLS}   ) = N \sigma^2/ (N_1 N_0 ) ,
$ 
can be estimated by 
$$ 
\widehat{V}_{OLS}=  \frac{N  (N_1 - 1)}{(N-2)N_1 N_0} s_1^2 +  \frac{N  (N_0 - 1)}{(N-2)N_1 N_0} s_0^2
 \approx  \frac{ s_1^2}{N_0}  +  \frac{  s_0^2}{N_1}.
$$
It is different from Neyman's variance estimator unless $N_1=N_0.$ Fortunately, we can avoid this problem by using Huber--White heteroskedasticity-robust variance estimator:
\begin{eqnarray*}
\widehat{V}_{HW}  =  \frac{   \sumN \widehat{\varepsilon}_i^2 (T_i - \bar{T})^2   }{  \left\{     \sumN  (T_i - \bar{T})^2  \right\}^2   } 
=  \frac{s_1^2}{N_1} \frac{N_1 - 1}{N_1} + \frac{s_0^2}{N_0} \frac{N_0-1}{N_0}
\approx  \frac{s_1^2}{N_1} + \frac{s_0^2}{N_0},
\end{eqnarray*}
which is asymptotically equivalent to the Neymanian variance estimator.
Therefore, the Wald statistic using $\widehat{V}_{HW}$ for testing $H_0(LM)$ is
asymptotically the same as the Neymanian test.

\subsection{Rao's Score Test and the FRT}

While the connection between the behavior of the Wald test for $H_0(LM)$ and Neyman's test has been established in previous studies, we make
a similar connection between Rao's score test for $H_0(LM)$ and the FRT in the following theorem.
\begin{theorem}
\label{thm::rao}
Rao's score test for $H_0(LM)$ under model (\ref{eq::LM}) is equivalent to
$$
  \frac{\widehat{\tau}}{ \sqrt{  \widehat{V}_S  } } \convD \mathcal{N}(0, 1),
$$
where $\widehat{V}_S = (N-1)s^2 /(N_1 N_0)  $.
\end{theorem} 
Ignoring the difference between $(N-1)$ and $N$ when $N$ is large, the difference between
$\widehat{V}_S$ and $\widehat{V}( \text{Fisher} )$ is of higher order, and
 Rao's score test is asymptotically equivalent to the FRT.
The sharp null hypothesis imposes the equal variance assumption on potential outcomes under treatment and control, leading to the equivalence of Rao's score test under the homoskedastic model and the FRT.

\begin{proof}[Proof of Theorem \ref{thm::rao}.]
The log likelihood function for the linear model in is
$$
l(\alpha, \beta, \sigma^2) = -\frac{N}{2}\log(2\pi\sigma^2) - \frac{\sumN (Y_i^{\obs} - \alpha - \beta T_i)^2 }{2\sigma^2}.
$$
Therefore, 
the score functions are 
\begin{eqnarray*}
\partial l/\partial \alpha  &=& \sumN (Y_i - \alpha - \beta T_i)/\sigma^2, \\
\partial l /\partial \beta   &=&  \sumN (Y_i - \alpha - \beta T_i)T_i /\sigma^2 , \\
\partial l /\partial \sigma^2 &=& -N/ (2\sigma^2)  + \sumN (Y_i - \alpha - \beta T_i)^2 /\{ 2  (\sigma^2 ) ^2\}  . 
\end{eqnarray*}
Plugging the MLEs under the null hypothesis with $\beta=0$, $\widetilde{\alpha} = \bar{Y}^{\obs}$ and $\widetilde{\sigma}^2 = \sumN (Y_i^{\obs} - \bar{Y}^{\obs})^2/N$ into the score functions, we obtain that only the second component of the score functions is non-zero:
$ \sumN (Y_i - \bar{Y}) T_i /\widetilde{\sigma}^2 = N_1 N_0 \widehat{\tau} / ( N\widetilde{\sigma}^2 )   .$

The second order derivatives of the log likelihood function are
\begin{eqnarray*}
\partial ^2 l /\partial \alpha^2  &=&  -N/\sigma^2 , \\
\partial ^2 l/\partial \beta^2     &=&  \sumN T_i^2/\sigma^2 = -N_1/\sigma^2,  \\
\partial ^2l/\partial (\sigma^2)^2 &=&  N / ( 2 \sigma^4 ) -  \sumN(Y_i - \alpha - \beta T_i)^2 / \sigma^6 , \\
\partial ^2l/\partial \alpha\partial \beta &=& -N_1/\sigma^2 , \\
\partial ^2 l/\partial \alpha\partial  \sigma^2 &=&  -  \sumN (Y_i - \alpha - \beta T_i)    / \sigma^4 , \\
\partial ^2l /\partial \beta\partial \sigma^2    &=&  -   \sumN (Y_i - \alpha - \beta T_i)T_i    / \sigma^4 .
\end{eqnarray*}
Therefore, the expected Fisher information matrix is
\begin{eqnarray*}
\bm{I}_N = 
\begin{pmatrix}
N/\sigma^2 & N_1/\sigma^2 & 0 \\
N_1/\sigma^2 & N_1/\sigma^2 & 0 \\
0& 0& N/ (  2  \sigma^4 ) 
\end{pmatrix},
\end{eqnarray*}
with the $(2,2)$-th element of $\bm{I}_N^{-1}$ being $N\sigma^2/(N_1N_0)$. Thus, Rao's score test for $H_0(LM)$ is
$$
\left( { N_1 N_0 \widehat{\tau} \over  N\widetilde{\sigma}^2 } \right) ^2 { N\widetilde{\sigma}^2 \over N_1N_0 }  \convD  \chi^2(1),
$$
or equivalently,
$$
\widehat{\tau} \Big/  \sqrt{  \frac{N\widetilde{\sigma}^2 }{N_1 N_0} }  = \widehat{\tau} \Big/  \sqrt{  \frac{(N-1)s^2 }{N_1 N_0} } 
 = \frac{\widehat{\tau}}{  \sqrt{  \widehat{V}_S }}
\convD  \mathcal{N}(0,1).
$$
\end{proof}

%
%
%
%

\section{More Details About Figure 4}
According to Corollary 1 in the main text, the Neymanian test has larger asymptotic power than the Fisherian test if and only if
$$
\left( \frac{1}{1-r} - \frac{1}{r} \right) \left\{   p_1(1-p_1) - p_0(1-p_0)  \right\}
+ (p_1 - p_0)^2 > 0.
$$
After some simple algebra, we can simplify the above inequality as
$$
(p_1 - p_0) (  a p_1 + b p_0 + c  ) > 0,
$$
where 
$$
a = \frac{ 1-r-r^2 }{ (1-r)r }, \quad
b = \frac{ 1-3r+r^2  }{ (1-r)r  },\quad 
c = \frac{  2r - 1}{  (1-r)r }.
$$
The shape of the region depends on the signs of $a$ and $b$, because the line $a p_1 + b p_0 +c =0$ intersects with the line $p_1-p_0=0$ at the point $(p_1, p_0) = (1/2,1/2).$ It is easy to show that $a>0$ if and only if 
$
0\leq r \leq \Gamma,
$
and $b>0$ if and only if
$
0\leq r \leq 1 - \Gamma,
$
where $\Gamma = (-1+\sqrt{5} ) / 2 \approx 0.618 $ is the reciprocal of the golden ratio. 
Therefore, when $r>1/2$, the region may have two shapes according the value of $r$ compared to $\Gamma$, as shown in Figure 4 of the main text. By symmetry, we can also plot the region when $r<1/2.$

\section{Other Test Statistics}

Consider a finite population of size $N=200$, and balanced completely randomized experiments.
Under the sharp null hypothesis, we generate potential outcomes $Y_i(1)=Y_i(0)$ from $\mathcal{N}(0,1)$; under the average null hypothesis, we generate $Y_i(1)$ from $\mathcal{N}(0,1)$, and generate $Y_i(0)$ as the order statistics of $Y_i(1)$. Clearly, the marginal distributions are the same but the correlation of the potential outcomes are different under different null hypothesis. 

The grey histogram in Figure \ref{fg::ks} is the randomization distribution of the Kolmogorov--Smirnov statistic under the sharp null hypothesis, and the white histogram with border is the randomization distribution under the average null hypothesis. The former is more disperse than the latter, indicating that the FRT using the Kolmogorov--Smirnov statistic tends to be conservative under the average null hypothesis.

The results for the Wilcoxon--Mann--Whitney rank sum statistic in Figure \ref{fg::wmw} are the same as above.

\bibliographystyle{Chicago}

\begin{figure}[ht]
\centering
\subfigure[Kolmogorov--Smirnov Statistic]{
    \includegraphics[width = 0.46\textwidth]{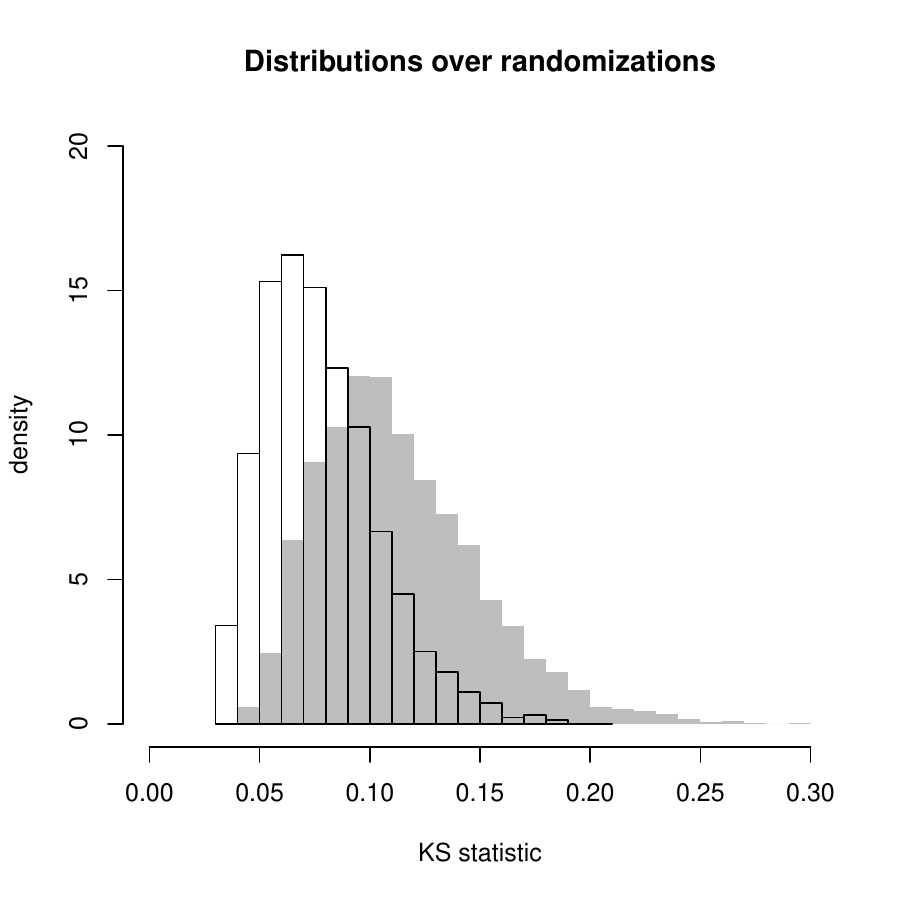}
    \label{fg::ks}
}
\subfigure[Wilcoxon--Mann--Whitney Rank Sum Statistic]{
    \includegraphics[width = 0.46\textwidth]{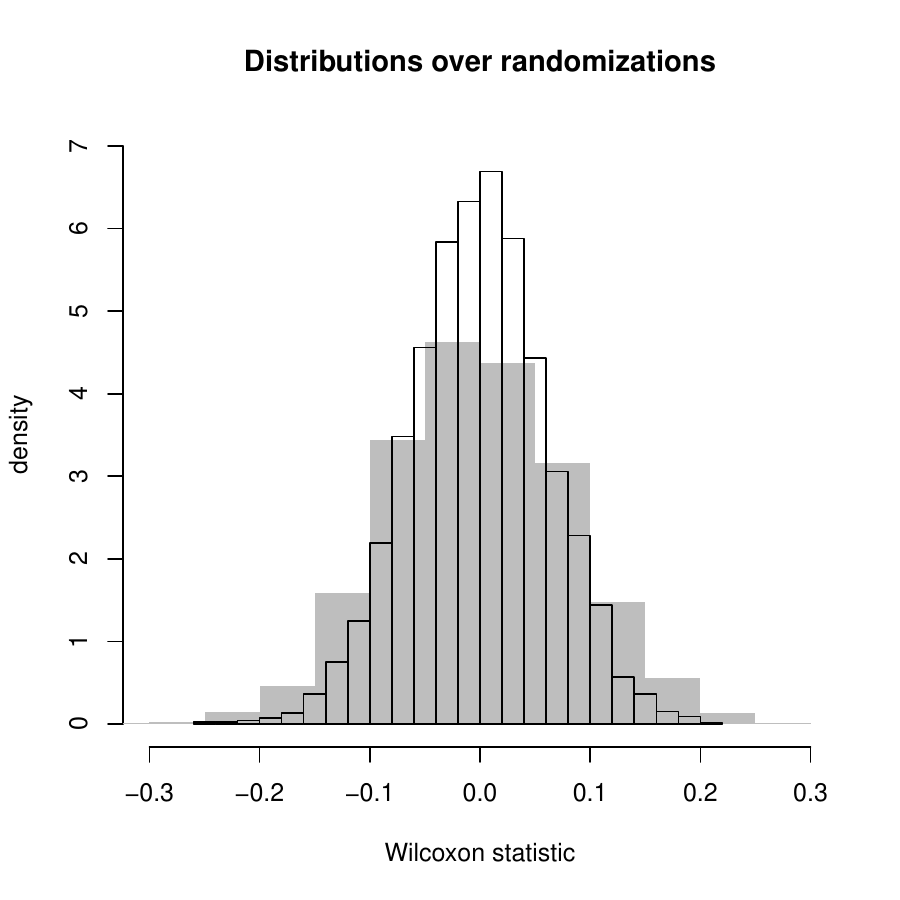}
    \label{fg::wmw}
}
\caption{Randomization Distributions of Different Test Statistics Under the Sharp Null (grey histograms) and Average Null (white histograms with borders). }
\label{fg::randomization-distribution}
\end{figure}

\end{document}